\documentclass[a4paper,11pt]{article}
\usepackage{graphicx}
\usepackage{amsmath,amsthm,amssymb,enumerate}%, esint}
\usepackage{euscript,mathrsfs}
\usepackage{dsfont}
\usepackage{xcolor}
\usepackage{empheq}
\usepackage{geometry}

\usepackage{color}
\catcode`\@=11 \@addtoreset{equation}{section}

\catcode`\@=12

\usepackage{stmaryrd}
\allowdisplaybreaks

\usepackage[labelformat=simple]{subcaption}

\usepackage{enumitem}
\usepackage{url}

\usepackage{tikz}
\usetikzlibrary{patterns}
\usepackage[normalem]{ulem}
\usepackage{cancel}

\newtheorem{Theorem}{Theorem}[section]
\newtheorem{Proposition}[Theorem]{Proposition}
\newtheorem{Lemma}[Theorem]{Lemma}
\newtheorem{Corollary}[Theorem]{Corollary}

\theoremstyle{definition}
\newtheorem{Definition}[Theorem]{Definition}

\newtheorem{Remark}[Theorem]{Remark}

\newcommand{\bTheorem}[1]{
\begin{Theorem} \label{T#1} }
\newcommand{\eT}{\end{Theorem}}

\newcommand{\bProposition}[1]{
\begin{Proposition} \label{P#1}}
\newcommand{\eP}{\end{Proposition}}

\newcommand{\bLemma}[1]{
\begin{Lemma} \label{L#1} }
\newcommand{\eL}{\end{Lemma}}
\newcommand{\bCorollary}[1]{
\begin{Corollary} \label{C#1} }
\newcommand{\eC}{\end{Corollary}}

\newcommand{\vrh}{\vr_h}
\newcommand{\vuh}{\vu_h}
\newcommand{\vth}{\vt_h}
\newcommand{\uih}{u_{i,h}}
\newcommand{\calDh}{\mathcal{B}_h}
\newcommand{\calP}{\mathcal{P}}

\newcommand{\Du}{\mathbf{D}(\vu)}
\newcommand{\Dhuh}{\bD_h(\vuh)}
\newcommand{\Dhuhk}{\bD_h(\vuh^k)}
\newcommand{\av}[1]{ \left\{ #1 \right\}}

\newcommand{\bRemark}[1]{
\begin{Remark} \label{R#1} }
\newcommand{\eR}{\end{Remark}}

\newcommand{\bDefinition}[1]{
\begin{Definition} \label{D#1} }
\newcommand{\eD}{\end{Definition}}

\newcommand{\Nu}{\mathcal{V}_{t,x}}

\newcommand{\intSh}[1] {\int_{\sigma} #1 \ds }

\newcommand{\order}{\mathcal{O}}

\newcommand{\bfu}{\mathbf{u}}
\newcommand{\bfv}{\mathbf{v}}
\newcommand{\bfq}{\mathbf{q}}

\newcommand{\bfe}{\mathbf{e}}

\newcommand{\bfx}{\mathbf{x}}

\newcommand{\bfphi}{\boldsymbol{\phi}}

\newcommand{\bfvarphi}{\boldsymbol{\varphi}}

\newcommand{\Pim}{\Pi_\mathcal{T}}
\newcommand{\Pie}{\Pi_\mathcal{E}}
\newcommand{\Pid}{\Pi_\mathcal{E}}

\newcommand{\Piei}{\Pi_\mathcal{E}^{(i)}}

\newcommand{\sumi}{\sum_{i=1}^d}

\newcommand{\ds}{\,{\rm d}S_x}

\newcommand{\bFormula}[1]{
\begin{equation} \label{#1}}
\newcommand{\eF}{\end{equation}}

\newcommand{\grid}{\mathcal{T}}
\newcommand{\dgrid}{\mathcal{G}}

\newcommand{\TS}{\Delta t}

\newcommand{\Divh}{{\rm div}_h}
\newcommand{\Gradh}{\nabla_h}

\newcommand{\Gradedge}{\nabla_\faces}

\newcommand{\Divmesh}{{\rm div}_{\mathcal{T}}}

\newcommand{\pdedge}{\eth_ \faces}
\newcommand{\pdedgei}{\eth_ \faces^{(i)}}

\newcommand{\pdmesh}{\eth_\grid}

\newcommand{\pdmeshi}{\eth_\grid}

\newcommand{\co}[2]{{\rm co}\{ #1 , #2 \}}
\newcommand{\Ov}[1]{\overline{#1}}

\newcommand{\aleq}{\stackrel{<}{\sim}}
\newcommand{\ageq}{\stackrel{>}{\sim}}

\newcommand{\vr}{\varrho}

\newcommand{\tvr}{\tilde \vr}
\newcommand{\tvu}{{\tilde \vu}}
\newcommand{\tvt}{\tilde \vt}

\newcommand{\vt}{\vartheta}
\newcommand{\vu}{\vc{u}}

\newcommand{\vn}{\vc{n}}

\newcommand{\vc}[1]{{\bf #1}}

\newcommand{\Div}{{\rm div}_x}
\newcommand{\Grad}{\nabla_x}

\newcommand{\dx}{\,{\rm d} {x}}

\newcommand{\dt}{\,{\rm d} t }

\newcommand{\jump}[1]{\left\llbracket#1\right\rrbracket}

\newcommand{\abs}[1]{| #1|}
\newcommand{\biggabs}[1]{\bigg| #1\bigg|}

\newcommand{\norm}[1]{\left\lVert#1\right\rVert}

\newcommand{\vU}{\vc{U}}

\newcommand{\dxdt}{\dx \ \dt}
\newcommand{\intO}[1]{\int_{\Omega} #1 \dx}

\newcommand{\intTO}[1]{\int_0^T \int_{\Omega} #1 \ \dxdt}

\newcommand{\vv}{\vc{v}}

\newcommand{\I}{\mathbb{I}}

\def\softd{{\leavevmode\setbox1=\hbox{d}%
          \hbox to 1.05\wd1{d\kern-0.4ex{\char039}\hss}}}%cstocs

\definecolor{Cgrey}{rgb}{0.85,0.85,0.85}
\definecolor{Cblue}{rgb}{0.50,0.85,0.85}
\definecolor{Cred}{rgb}{1,0,0}
\definecolor{fancy}{rgb}{0.10,0.85,0.10}
\definecolor{forestgreen}{rgb}{0.13, 0.55, 0.13}

\newcommand\Cbox[2]{%
    \newbox\contentbox%
    \newbox\bkgdbox%
    \setbox\contentbox\hbox to \hsize{%
        \vtop{
            \kern\columnsep
            \hbox to \hsize{%
                \kern\columnsep%
                \advance\hsize by -2\columnsep%
                \setlength{\textwidth}{\hsize}%
                \vbox{
                    \parskip=\baselineskip
                    \parindent=0bp
                    #2
                }%
                \kern\columnsep%
            }%
            \kern\columnsep%
        }%
    }%
    \setbox\bkgdbox\vbox{
        \color{#1}
        \hrule width  \wd\contentbox %
               height \ht\contentbox %
               depth  \dp\contentbox
        \color{black}
    }%
    \wd\bkgdbox=0bp%
    \vbox{\hbox to \hsize{\box\bkgdbox\box\contentbox}}%
    \vskip\baselineskip%
}

\date{}
\newcommand{\pd}{\partial}

\newcommand{\eps}{\varepsilon}
\newcommand{\faces}{\mathcal{E}}
\newcommand{\facesi}{\faces _i}
\newcommand{\facesK}{\faces(K)}
\newcommand{\facesKi}{\faces _i(K)}
\newcommand{\facesint}{\faces}
\newcommand{\facesinti}{\facesi}

\newcommand{\bQh}{ Q_h}
\newcommand{\bWh}{ W_h}

\newcommand{\bD}{\mathbf D}

\newcommand{\muh}{h^\eps}

\allowdisplaybreaks
\begin{document}

%%%%%%%%%%%%%%%%%%%%%%%%%%%%%%%%
%\title{Convergence of a finite volume scheme for the compressible Navier--Stokes--Fourier system}
\title{ On the convergence of a finite volume method\\ for the Navier-Stokes-Fourier system}
\author{Eduard Feireisl$^{\clubsuit,}$\thanks{E.F.,  H.M. and B.S. have received funding from
the Czech Sciences Foundation (GA\v CR), Grant Agreement
18--05974S.
The Mathematical Institute of the Czech Academy of Sciences is supported by RVO:67985840.
\newline
\hspace*{1em} $^\spadesuit$M.L. has been funded by the Deutsche Forschungsgemeinschaft (DFG, German Research Foundation) - Project number 233630050 - TRR 146 as well as by  TRR 165 Waves to Weather.}
\and M\'aria Luk\'a\v{c}ov\'a-Medvi\softd ov\'a $^{\spadesuit}$
\and Hana Mizerov\'a $^{*, \dagger}$
\and Bangwei She $^{*}$
}

\date{\today}

\maketitle

\centerline{$^*$ Institute of Mathematics of the Academy of Sciences of the Czech Republic}
\centerline{\v Zitn\' a 25, CZ-115 67 Praha 1, Czech Republic}
\centerline{feireisl@math.cas.cz, mizerova@math.cas.cz, she@math.cas.cz}

\bigskip
\centerline{$^\clubsuit$ Technische Universit\"at Berlin}
\centerline{Stra\ss e des 17.~Juni, Berlin, Germany}

\bigskip
\centerline{$^\spadesuit$ Institute of Mathematics, Johannes Gutenberg-University Mainz}
\centerline{Staudingerweg 9, 55 128 Mainz, Germany}
\centerline{lukacova@uni-mainz.de}

\bigskip
\centerline{$^\dagger$ Department of Mathematical Analysis and Numerical Mathematics}
\centerline{Faculty of Mathematics, Physics and Informatics of the Comenius University}

\centerline{Mlynsk\' a dolina, 842 48 Bratislava, Slovakia}

\begin{abstract}

We study convergence of a finite volume scheme for the Navier--Stokes--Fourier system describing the motion of  compressible
viscous and heat conducting fluids. The numerical flux  uses upwinding with an additional numerical diffusion of order $\mathcal O(h^{ \eps+1})$, $0<\varepsilon<1$.   The  approximate solutions are piecewise constant functions with respect to the underlying mesh. 
We show that any uniformly bounded sequence of numerical solutions converges unconditionally to the solution of the Navier--Stokes--Fourier system. In particular, the existence of the solution to the Navier--Stokes--Fourier system is not a priori assumed.

% Assuming {\cmag uniform boundedness of the numerical solution discrete temperature and density, we show that}  numerical solutions generate,up to a subsequence, the Young measure that represents a dissipative measure--valued solution of the Navier--Stokes--Fourier system.  Using the dissipative measure--valued--strong uniqueness result we  {\cmag prove strong convergence of  the numerical solution towards the strong solution of the limit system.}
\end{abstract}


\tableofcontents

\section{Introduction}
 The time evolution of viscous compressible and heat conducting fluids is governed by the conservation of mass,  momentum and energy. Altogether these conservation
laws yield the well-known Navier--Stokes--Fourier system
\begin{subequations}\label{ns_eqs}
\begin{equation}
\pd_t \vr  + \Div (\vr \vu ) = 0,
\end{equation}
\begin{equation}
\pd_t (\vr \vu)  + \Div (\vr \vu \otimes \vu  ) + \Grad p= \Div \mathbb{S}(\Du) ,
\end{equation}
\begin{equation}
\pd_t (\vr e)  +  \Div ( \vr e\vu) - \Div (\kappa \Grad \vt) =2\mu |\bD(\vu)|^2  + \lambda |\Div \vu|^2 -p \Div \vu ,
\end{equation}
\end{subequations}
where $\vr, \vu, \vt, p, e$ are the density, velocity,   temperature, pressure and internal energy, respectively.
The pressure $p$ satisfies the perfect gas law
\[ p = \vr  \vt, \] and the internal energy is  $$e=c_v\vt,$$
where $c_v > 0 $ is the specific heat at constant volume. The constant $\kappa >0$ denotes the heat conductivity coefficient. 
 Further, we have denoted by $$\bD(\vu)=\frac{\Grad \vu + \Grad^T \vu}{2}$$ the symmetric velocity gradient and by
$$\mathbb S(\Du)= 2\mu \bD(\vu)  + \lambda  \Div \vu \I $$ the viscous stress tensor with the viscosity coefficients $\mu>0$ and $ \lambda \geq 0$. 
System \eqref{ns_eqs} is solved in the time--space  cylinder $(0,T)\times\Omega$.  We prescribe the periodic boundary condition, which means  $\Omega \subset \mathbb R^d, \ d=2,3,$ is assumed to be a flat torus. 
To close the system we impose the initial conditions
\begin{equation*}\label{initial_data}
\vr (0) =\vr _0 , \; \vu (0) =\vu _0 ,\; \vt (0) =\vt _0 ,\; \text{ with } \vr_0 >0\text{ and } \vt_0>0.
\end{equation*}

System \eqref{ns_eqs} has numerous everyday applications, e.g., in aerodynamics, hydrodynamics, engineering or even in medicine.  Therefore its numerical approximations have been widely studied in the past decades. Let us mention a few well-established and practical schemes, e.g.,
\cite{BenArtzi_GRP, Cockburn_DG, DoFei_DG, Godunov, Lukacova_FVEG, Shen_CESE, Shu_ENO, Toro, Xu_BGK}. Despite of such  variety of efficient numerical schemes, their convergence analysis  is still open in general.  Though there are some convergence (and even error estimate) results for numerical methods for the isentropic Navier-Stokes equations, see, e.g., \cite{GalHerMalNov, GallouetMAC, HosekShe_MAC, jovanovic, Karper} or \cite{FHMN, FL_18,FVNS_FLMS}, 
the convergence analysis of the full Navier--Stokes--Fourier system is considerably more involved and  much less results are available in the literature.
For a mixed finite element--finite volume method based on the
Crouzeix--Raviart finite elements Feireisl, Karper and Novotn\'y \cite{FKN2016} proved the convergence to a weak solution for a rather specific state equation $p = a \rho^\gamma + b \rho + \rho \theta$, $a,b >0$ and $\gamma > 3.$  It is to be pointed out that the generalization of the result obtained in \cite{FKN2016} to other schemes is still open, cf.~also \cite{HS_NSF}. On the other hand, in our recent works \cite{FL_18, FLM18, FLM18_brenner, FVNS_FLMS} we have proposed a new, rather general way for the convergence analysis via the concept of
dissipative measure--valued (DMV) solutions.

Our approach bears some similarities with the recent works of Fjordholm et al.~\cite{FjKaMiTa, FjMiTa2, FjMiTa1}, who studied the convergence of
entropy stable 
finite volume schemes to a measure--valued solution of the Euler equations. The main difference  in using the concept of DMV  solutions lies in the fact that we relax the energy conservation asking only  the global energy to dissipate over  time. Similarly to Fjordholm et al. we also require that the entropy inequality holds, cf.~Definition~\ref{def_DMV}. 

The main goal of this paper is to demonstrate that the strategy proposed in \cite{FL_18, FVNS_FLMS} can be extended to obtain the convergence for the full Navier--Stokes--Fourier system \eqref{ns_eqs}.
To solve  the latter  numerically we apply a finite volume scheme  with the numerical flux function based on upwinding to get a piecewise constant approximation of all unknown quantities.  Under a realistic assumption that the numerical solutions have bounded temperature and density, we can prove the consistency of the finite volume scheme. This fact together with some suitable {\sl a  priori} estimates implies that
the sequence of numerical solutions generates, up to a subsequence, a DMV solution. Note, that in contrast to the isentropic Navier--Stokes equations, we need to control also the gradients of the velocity and temperature, since they are
now included in the  support of the corresponding Young measure, cf.~\cite{stability_nsf}. 
Furthermore, using the DMV-strong uniqueness principle for the solution of the Navier--Stokes--Fourier system, cf.~\cite{stability_nsf}, we get the strong convergence of the piecewise constant solutions to the strong (classical) solution on its lifespan, see
Theorem~\ref{thm_dmvs}.  For any uniformly bounded sequence of numerical solutions  we also obtain the global in time convergence
to the strong (classical) solution of the Navier--Stokes--Fourier system (\ref{ns_eqs}) without a priori assuming the existence of its solution, see Theorem~\ref{thm_convergence}.
Here \emph{strong}  means solutions in the standard energy spaces used by Valli and Zajackowski  \cite{Valli}. In particular, as shown in \cite{Valli} these are \emph{classical} solutions in the sense that all necessary derivatives are continuous.

The rest of the paper is organized as follows. In Section 2 we introduce the necessary notations and the numerical scheme. In Section 3 we show that the  discrete solutions satisfy the global energy dissipation  and the entropy inequality. The consistency formulation of the scheme is proved in Section~4.
We present the main results on the convergence of our finite volume scheme in  Section~5. 
%Section~6 is devoted to the numerical performance of the scheme in one and two space dimensions.

\section{Numerical scheme}\label{sec_Notations}
In this section we collect the necessary  apparatus of the numerical analysis and introduce the finite volume method for the
Navier--Stokes-Fourier system (\ref{ns_eqs}).

\subsection{Space discretization}

{\bf Mesh.}  Let $\grid$ be a uniform quadrilateral mesh
 such that
%of a domain $\Omega$ ,
\[
 \Omega = \bigcup_{K \in \grid} K,
\]
where $K$ is a square $(d=2)$ or a cube $(d=3)$.  For any $K\in\grid$ we denote by $\bfx_K$  its center of mass and by $|K|=h^d$  its volume.
Let $\faces$ be the set of all faces, and $\facesi, \; i = 1,\ldots, d$, be the set of all faces that are orthogonal to the unit vector $\bfe_{i}$ of the $i^{th}$ canonical direction. Moreover, we write $\facesK$ as the set of all faces of  an element $K$ and $\facesKi = \facesK \cap \facesi$.
For any $\sigma$ being the common face of elements $K$ and $L$, we write $\sigma= K|L$. We further write $\sigma =\overrightarrow{K|L}$ if $\bfx_L = \bfx_K + h\bfe_i$ for  any $i = 1,\ldots,d $.
By $\bfx_\sigma$  we denote the center of mass of a generic face $\sigma$ and  by $|\sigma| = h^{d-1}$  its Lebesque measure.

\noindent {\bf Function space.}
 The symbol $Q_h$ stands for the set of piecewise constant functions on primary grid $\grid$.
 We approximate the  density, velocity and temperature  by discrete functions $\vrh$,  $\vuh,$ $\vth\in Q_h$, respectively.  Analogously, $s_h=s(\vrh,\vth)$ stands for a piecewise constant approximation of a function $s=s(\vr,\vt)$ with respect to $\grid.$ Note that hereafter $\vv_h\in Q_h$ means that every component of a vector--valued function $\vv_h$ belongs to the set $Q_h.$

The standard projection operator associated to $Q_h$ reads
\begin{equation*}
\Pim: \, L^1(\Omega) \rightarrow Q_h. \quad
\Pim  \phi  = \sum_{K \in \grid} 1_{K} \frac{1}{|K|} \int_K \phi \dx .
\end{equation*}
For any $v_h \in Q_h$ we have
\[ \intO{v_h} = \sum_{K\in\grid}|K| v_K, \quad v_K = v_h|_K.
\]
Further, we use the following notations for the average and jump operators
\[
\Ov{v}(x) = \frac{v^{\rm in}(x) + v^{\rm out}(x) }{2},\
\jump{ v }  = v^{\rm out}(x) - v^{\rm in}(x), \ \mbox{ where } \
v^{\rm out}(x) = \lim_{\delta \to 0+} v(x + \delta \vc{n}),\
v^{\rm in}(x) = \lim_{\delta \to 0+} v(x - \delta \vc{n}),\
\]
whenever $x \in \sigma \in \facesint$.

\noindent {\bf Discrete operators.} For piecewise constant functions  we define the discrete gradient and divergence operators in the following way
\begin{equation*}
\begin{aligned}
\Gradh r_h(\bfx) & =  \sum_{K \in \grid} \left( \Gradh r_h\right)_K  1_K, \quad&
\left( \Gradh r_h\right)_K  &=  \frac{|\sigma|}{|K|} \sum_{\sigma \in \facesK} \Ov{r_h} \vc{n} , \quad
\\
\Divh \bfv_h(\bfx) &= \left( \Divh \bfv_h\right)_K 1_K, \quad&
\left( \Divh \bfv_h\right)_K &= \left( \Gradh \cdot  \bfv_h\right)_K  =  \frac{|\sigma|}{|K|} \sum_{\sigma \in \facesK} \Ov{\bfv_h} \cdot \vc{n} ,
\\
\Gradh \bfv_h &= \left( \Gradh v_{1,h}, \ldots,   \Gradh v_{d,h}\right)^T , \quad&
\bD_h (\bfv_h) &= \left(\Gradh \bfv_h   +\Gradh^T \bfv_h \right)/2,
\end{aligned}
\end{equation*}
for any $r_h, \; \bfv_h \in Q_h$. It is worth mentioning that due to the fact that $ \int_{\pd K}  \vc{n}\ds =0 $, we have
\[ \int_{\pd K} \Ov{r_h} \vc{n}\ds = \frac12 \int_{\pd K} \jump{r_h} \vc{n}\ds  .
\]
The discrete Laplace operator can be defined analogously
\begin{equation*}\label{dis_laplace}
\Delta_h r_h(\bfx) =  \sum_{K \in \grid} \left(\Delta_h r_h\right)_K  1_K, \quad
\left(\Delta_h r_h\right)_K =  \frac{|\sigma|}{|K|} \sum_{\sigma \in \facesK} \frac{ \jump{r_h} }{h}, \qquad r_h \in Q_h.
\end{equation*}

%In order to define the discrete divergence operator, it is convenient to define a function space based on the element edges.
In what follows we will also work with functions  evaluated at the cell faces.  Therefore it is convenient to introduce a dual grid
associated to  faces $\sigma$ and the corresponding discrete  function space.
% acting on the dual grid.

\noindent {\bf Dual grid.} For any $\sigma=K|L\in \facesint$, we define a dual  cell  $D_\sigma := D_{\sigma,K} \cup D_{\sigma,L}$, where $D_{\sigma,K}$ (resp.~$D_{\sigma,L}$) is half of  an element $K$ (resp.~$L$), see Figure~\ref{fig:mesh} for an example of such a cell in two dimensions.   We denote the set of all dual cells by $\dgrid$.
Furthermore, we define $\dgrid_i = \{ D_\sigma \}_{\sigma \in \facesi}, i = 1,\ldots, d$.
\begin{figure}[!h]
\centering
\begin{tikzpicture}[scale=1.0]
\draw[-,very thick](0,-2)--(5,-2)--(5,2)--(0,2)--(0,-2)--(-5,-2)--(-5,2)--(0,2);
%\draw[-,very thick](0,-2)--(0,2);
%\draw[fill=blue!10](0,-2)--(0,2)--(-2.5,2)--(-2.5,-2)--(0,-2);
%\draw[fill=green!30, pattern=northeast](0,-2)--(2.5,-2)--(2.5,2)--(0,2)--(0,-2);
\draw[-,very thick, green=90!, pattern=north west lines, pattern color=green!30] (0,-2)--(2.5,-2)--(2.5,2)--(0,2)--(0,-2);
\draw[-,very thick, blue=90!, pattern= north east  lines, pattern color=blue!30] (0,-2)--(0,2)--(-2.5,2)--(-2.5,-2)--(0,-2);

\path node at (-3.5,0) { $K$};
\path node at (3.5,0)  { $L$};
\path node at (-2.5,0) {$ \bullet$};
\path node at (-2.8,-0.3) {$ \bfx_K$};
\path node at (2.5,0) {$\bullet$};
\path node at (2.7,-0.3) {$ \bfx_L$};
\path node at (0,0) {$\bullet$};
\path node at (0.3,-0.3) {$ \bfx_\sigma$};

\path (-0.4,0.8) node[rotate=90] { $\sigma=\overrightarrow{K|L}$};
 \path (-1.5,1.4) node[] { $D_{\sigma,K}$};
 \path (1.5,1.4) node[] { $D_{\sigma,L}$};
 \end{tikzpicture}
\caption{Dual grid}
 \label{fig:mesh}
\end{figure}
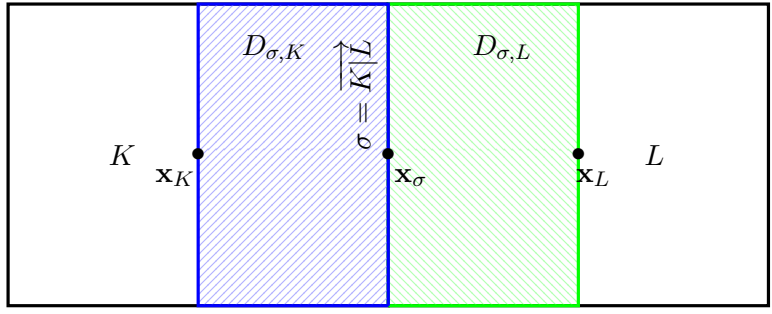
Now we are able to define $W_h^{(i)},$ $i = 1,\ldots,d,$ as the space of piecewise constant functions on the dual grid $\dgrid _i$. By $\bfq=(q_{1}, \ldots, q_{d}) \in \bWh:=\big( W_h^{(1)}, \ldots,  W_h^{(d)} \big)$ we mean that $q_{i} \in W_h^{(i)}$, for all $i=1, \ldots, d$.

\noindent Accordingly, the associated projection of the functional spaces  $W_h $ is given by
\begin{equation*}
\Pid: \quad L^1(\Omega) \rightarrow W_h, \quad
\Pid  = (\Pid^{(1)}, \ldots, \Pid^{(d)}), \quad \Pid^{(i)} \phi
= \sum_{\sigma \in \facesi}  \frac{1_{D_\sigma}}{|D_\sigma|} \int_{D_\sigma} \phi \dx .
\end{equation*}

\noindent For any $r_h\in Q_h$  and $\bfq_h =(q_{1,h}, \ldots, q_{d,h}) \in W_h$ we define the following standard difference operators
\begin{equation*}
\begin{aligned}
 \pdedgei r_h(\bfx) & = \sum_{\sigma\in\faces} 1_{D_\sigma} \left(\pdedgei r_h \right)_{\sigma}, \quad &
\left(\pdedgei r_h\right) _{\sigma} &=\frac{r_h|_L - r_h|_K}{ h} \; \text{ for any } \; \sigma =\overrightarrow{K|L} \in \facesinti,\\
 \pdmeshi q_{ i,h}(\bfx)&  = \sum_{K \in \grid} \left(\pdmeshi q_{i,h}\right)_{K}1_K, \quad &
\left(\pdmeshi q_{i,h} \right)_{K} &=\frac{ q_{i,h}|_{\sigma'} - q_{i,h}|_{\sigma}}{h}  \; \mbox{ for all } \; \sigma, \sigma' \in \facesKi \mbox{ and } \bfx_{\sigma'} =\bfx_\sigma +h \bfe_i .
\end{aligned}
\end{equation*}
%and extend them to the whole domain
%\begin{equation}
%\pdedgei r_h(\bfx) = \sum_{\sigma\in\faces} 1_{D_\sigma} \left(\pdedgei r_h \right)_{\sigma}, \quad
%%%\eth_i r(\bfx)= \sum_{K\in \grid} 1_K \left(\eth_i r \right)_K = \Pim{\pdedgei r }
%\pdmeshi q_{ i,h}(\bfx) := \sum_{K \in \grid} \left(\pdmeshi q_{i,h}\right)_{K}1_K,\; \ i = 1,\ldots, d.
%\end{equation}
With the above notations, we further define
\begin{equation}\label{grad_edge}
 \begin{split}
& \Gradedge r_h = \left(\pdedge^{(1)} ,\ldots, \pdedge^{(d)} \right) r_h , \quad
 \Divmesh \bfq_h = \sumi \pdmeshi q_{i,h} .
 % \Divh \vuh = \sumi \pdmeshi \Piei \uih, \quad
%\\ &
%\Gradh r_h =\pdmesh \Pie r_h = \left(  \pdmesh \Pie^{(1)} , \ldots, \pdmesh \Pie^{(d)}  \right)r_h, \quad
% (\Gradh \vuh) = (\Gradh u_{1,h}, \ldots, u_{d,h})^T, \quad
% \Dhuh =\frac12 (\Gradh \vuh +\Gradh^T \vuh) .
\end{split}
\end{equation}
It is easy to observe that
\begin{equation} \label{avg_dif_order}
%(\Gradh r_h)_K = \frac{1}{|K|} \int_{\pd K} \Ov{r_h} \vc{n} =  \frac{1}{|K|} \int_{\pd K} \frac{ \jump{r_h}}{2} \vc{n}, \quad
\pdmeshi \Piei r_h = \Pim \pdedgei r_h, \quad \Gradh r_h = \pdmesh \Pie r_h = \Pim \Gradedge r_h , \quad \Delta_h \vth = \Divmesh \Gradedge \vth .
\end{equation}

\noindent {\bf Integration by parts.}
 Let us start with recalling the following algebraic identity
\begin{equation*}\label{avg_diff}
\Ov{u_h v_h} - \Ov{u_h}\ \Ov{v_h} =\frac14 \jump{u_h} \jump{v_h}
\end{equation*}
together with the product rule
\begin{equation}\label{product_rule}
\jump{ u_h v_h }  = \Ov{u_h} \jump{v_h}  + \jump{u_h}  \Ov{v_h}\; ,
\end{equation}
which are valid for any $u_h, v_h\, \in Q_h.$
A direct application of the product rule \eqref{product_rule} further implies
%\begin{equation*}\label{basic_eq1}
%\jump{r_h \bfv_h} \jump{\bfv_h} - \frac12 \jump{r_h} %\jump{|\bfv_h|^2} =\Ov{r_h}\jump{\bfv_h}^2
% \mbox{ for }  r_h,\ \bfv_h \in \bQh,
%\end{equation*}
%and  
the following lemma.
\begin{Lemma}\cite[Lemma 2.2]{FVNS_FLMS}  For  any  $ r_h,$ $ \bfv_h \in \bQh$ it holds that
\begin{equation}\label{basic_eq2}
 \sum_{ \sigma \in \facesint } \intSh{ \left( \Ov{r_h} \jump{ \bfv_h } +\Ov{\bfv_h }  \jump{r_h  }   \right) \cdot \vc{n} }
=0.
\end{equation}
\end{Lemma}
\noindent Indeed, \eqref{basic_eq2} indicates the Grad--Div duality for any $r_h , \bfv_h \in Q_h$ , i.e.,
\begin{equation*}
\begin{split}
&\intO{\Gradh r_h \cdot \bfv_h} = \sum_K \bfv_K \cdot \int_{\pd K} \Ov{r_h}  \vc{n}  \ds
= \sum_{K \in \grid} \bfv_K \cdot \int_{\pd K} \frac{\jump{r_h}}{2}  \vc{n} \ds
= \sum_{\sigma \in \faces} \intSh{ \Ov{\bfv_h} \cdot  (\jump{r_h} \vc{n}) }
\\&= - \sum_{\sigma \in \faces} \intSh{ \jump{\bfv_h} \cdot  (\Ov{r_h} \vc{n}) }
=  - \sum_{K \in \grid} r_K  \int_{\pd K} \frac{\jump{\bfv_h}}{2}  \cdot\vc{n} \ds
=  - \sum_{K \in \grid} r_K  \int_{\pd K} \Ov{\bfv_h} \cdot  \vc{n} \ds
\\&= - \intO{r_h \Divh \bfv_h} .
\end{split}
\end{equation*}
It is also easy to observe the following discrete integration by parts formulae for all $r_h, \phi_h \in Q_h$ and $ \bfq_h \in W_h$
\begin{subequations}\label{int_by_part}
\begin{equation*}\label{int_by_part_las}
\intO{\Delta_h r_h \phi_h}
=-\intO{ \Gradedge r_h \cdot \Gradedge \phi_h }= \intO{ r_h \Delta_h \phi_h},
\end{equation*}
\begin{equation*}\label{int_by_part_gradm}
\intO{q_{i,h} \pdedgei r_h} = - \intO{  r_h \pdmeshi q_{i,h}},\;\mbox{ for all }\; i = 1,\ldots,d.
\end{equation*}
%We note that \eqref{int_by_part_gradm} is just another form of \eqref{basic_eq2}.
\end{subequations}

%Noticing $\eth_j \Piei u_i = \Piei \pdedgej  u_i$, and thanks to the inequality $\left(\frac{a+b}{2}\right)^2 \leq \frac{a^2+b^2}{2}$, we have
%\begin{equation}\label{gradient_inequ}
%\intO{|\Gradh \vuh|^2} \leq \intO{|\Gradedge \vuh|^2} .
%\end{equation}
%For any $r,v \in Q_h$ and $i\in \{ 1,\ldots, d\}$, it is easy to derive an intergration by parts formula
%\begin{subequations}
%\begin{equation}\label{int_by_part_1}
%\intO{\eth_i r v} = - \intO{r \eth_i v},
%\end{equation}
%which further indicates
%\begin{equation}\label{int_by_part_2}
%\intO{\eth_i r \eth_j v  } = -\intO{ r \eth_i \eth_j v  }
%= -\intO{ r \eth_j \eth_i  v  }
%=\intO{\eth_j r \eth_i v  } .
%\end{equation}
%\end{subequations}

\noindent {\bf Useful estimates.}
Next, we list some basic inequalities used  in the numerical analysis.
We assume the reader is fairly familiar with this matter, for which we refer to the monograph \cite{EyGaHe}, and the article  \cite{GallouetMAC}. If
$\phi \in C^1(\Omega)$, $h \in (0, h_0), $ $h_0 \ll 1$, we have
\begin{equation} \label{n4c}
  \Big|  \jump{ \Pim  \phi  }   \Big|_{\sigma} \aleq h \| \phi \|_{C^1},  \mbox{ for any } x \in \sigma \in \facesint, \;
\norm{\phi - \Pim  \phi }_{L^p(\Omega)} \aleq h \| \phi \|_{C^1}, \;
\norm{\Pim \phi - \Pie \Pim  \phi }_{L^p(\Omega)} \aleq h \| \phi \|_{C^1}.
\end{equation}
Here and hereafter we  denote $A \aleq B$ if $A \leq cB$ for a positive constant $c$ which is independent of  the discretization parameter $h.$
Furthermore, if $\phi \in C^2(\Omega)$  we have for all $1< p \leq \infty$, $h \in (0, h_0), $ $h_0 \ll 1$
\begin{equation} \label{n4c2}
\begin{aligned}
\norm{ \Grad \phi - \Gradedge \big(\Pim \phi\big)  }_{L^p(\Omega)}  \aleq h, \quad
%\norm{ \Grad \phi - \Pim  \Gradedge \big(\Pim \phi\big)  }_{L^p}  \aleq h, \quad
%\\
\norm{ \Grad \phi - \Gradh \big(\Pim \phi\big)  }_{L^p(\Omega)}  \aleq h, \quad
\norm{ \Div \phi - \Divh  (\Pim \phi )   }_{L^p(\Omega)}  \lesssim h.
\end{aligned}
\end{equation}

%\noindent Next, we introduce a diffusive upwind flux.
\paragraph{Diffusive upwind flux.}
For a given a velocity  $\vu_h \in \bQh$ and  a quantity $r_h \in Q_h$
the upwind numerical flux is defined at each face  $\sigma \in \facesint$ as

\begin{align*}\label{Up}
Up [r_h, \vu_h]   =r_h^{\rm up} \vu_h\cdot \vc{n}
=r_h^{\rm in} [\Ov{\vu_h} \cdot \vc{n}]^+ + r_h^{\rm out} [\Ov{\vu_h} \cdot \vc{n}]^-
= \Ov{r}_h \ \Ov{\vu_h} \cdot \vc{n} - \frac{1}{2} |\Ov{\vu_h} \cdot \vc{n}| \jump{r_h},
\end{align*}
where
\begin{equation*}
[f]^{\pm} := \frac{f\pm |f|}{2} \quad \mbox{and} \quad
r^{\rm up} :=
\begin{cases}
 r^{\rm in} & \mbox{if} \ \Ov{\vu} \cdot \vc{n} \geq 0, \\
r^{\rm out} & \mbox{if} \ \Ov{\vu} \cdot \vc{n} < 0.
\end{cases}
\end{equation*}
Now, we can define a numerical flux function
\begin{equation}\label{num_flux}
F_h (r_h,\vu_h)
={Up}[r_h, \vu_h] - \muh \jump{ r_h }
,\quad 0< \eps <1.
\end{equation}
 Let us point out that the $\muh-$term introduced in the numerical flux actually acts as an artificial diffusion term  of order $\mathcal{O}(h^{\eps+1})$ in our finite volume scheme \eqref{scheme_ns_fv} defined below.  Indeed,
\[ \sum_{\sigma\in\facesK}   \frac{|\sigma|}{|K|} \muh\jump{r_h}  =  h^{\eps+1} (\Delta_h r_h)_K
\]
for $r_h \in \{\vrh,  \vrh\vuh, \vrh \vth\}$. 
Note that the vector--valued flux function $\mathbf{F}_h  (r_h,u_h)$ that is used in the momentum equation with $r_h = \vrh \vuh$ is defined componentwisely.

\subsection{Time discretization}
For a given time step $\TS \approx h>0$
we denote the approximation of a function $v_h$ at time $t^k= k\TS$ by $v_h^k$ for $k=1,\ldots,N_T(=T/\TS)$.  The time derivative is approximated by the backward finite difference
\[
  D_t v_h^k = \frac{v_h^k- v_h^{k-1}}{\TS},\ \mbox{ for } k=1,2,\ldots, N_T.
\]
Furthermore, we introduce  the functions $(\vrh,\vuh,\vth),$ piecewise constant in time, which are  given by
\begin{equation*}
\begin{aligned}
&\vrh(t, \cdot) =\vrh^0 \mbox{ for }  t \in [0,\TS),\quad  \vrh(t,\cdot)=\vrh^k \mbox{ for } t\in [k\TS,(k+1)\TS),\quad k=1,2,\ldots,N_T,\\
&\vuh(t,\cdot) =\vuh^0 \mbox{ for }  t \in [0,\TS),\quad  \vuh(t,\cdot)=\vuh^k \mbox{ for } t\in [k\TS,(k+1)\TS),\quad k=1,2,\ldots,N_T,\\
&\vth(t,\cdot) =\vth^0 \mbox{ for }  t \in [0,\TS),\quad  \vth(t,\cdot)=\vth^k \mbox{ for } t\in [k\TS,(k+1)\TS),\quad k=1,2,\ldots,N_T,\\
 \mbox{ and } & p_h(t)=p(\vrh(t)),  \quad s_h(t)= s(\vrh(t), \vth(t)).
\end{aligned}
\end{equation*}
The discrete time derivative then reads
\[
 D_t v_h = \frac{v_h (t,\cdot) - v_h(t - \Delta t,\cdot)}{\TS} .
\]

\subsection{Numerical method for the  Navier--Stokes--Fourier system}
We  are now ready to propose the following finite volume scheme for the compressible Navier--Stokes-Fourier system \eqref{ns_eqs}.

\begin{Definition}[Finite volume scheme]
  Given the initial values  $(\vrh^0,\vuh^0, \vth^0) =(\Pim\vr_0, \Pim\vu_0, \Pim \vt_0)$, we  seek a solution
$\{(\vrh^k,\vuh^k, \vth^k)\}_{k=1}^{N_T} \in Q_h \times Q_h \times Q_h$ satisfying, for all $K\in\grid,$
\begin{subequations}\label{scheme_ns_fv}
\begin{align}
D_t \vr^k _K &+ \sum_{\sigma \in \facesK} \frac{|\sigma|}{|K|} F_h (\vrh^k,\vuh^k) =0 ,\label{scheme_ns_den_fv} \\
D_t (\vrh^k \vuh^k)_K
%+ \sum_{\sigma \in \facesK} \frac{|\sigma|}{|K|} \left({\bf F}_h  (\vrh^k \vuh^k,\vuh^k)  + \Ov{p_h^k} \vc{n}- \lambda \Ov{\Divh \vuh^k} \vc{n}\right)
&+ \sum_{\sigma \in \facesK} \frac{|\sigma|}{|K|}  {\bf F}_h  (\vrh^k \vuh^k,\vuh^k)
+ \Gradh p_h^k  =   2 \mu (\Divh \Dhuhk)_K+\lambda \Gradh (\Divh \vuh^k), \label{scheme_ns_mom_fv}\\
c_v D_t (\vrh^k \vth^k)_K  &+ c_v \sum_{ \sigma \in \pd K } \frac{|\sigma|}{|K|} F_h (\vrh^k \vth^k, \vuh^k)
 -  \kappa \Delta_h \vth^k
 = 2 \mu |\Dhuhk|_K^2 + \lambda |\Divh \vuh^k|_K^2
 -p_K^k (\Divh \vuh^k)_K . \label{scheme_ns_ene_fv}
\end{align}
\end{subequations}
For convenience of analysis we rewrite the above finite volume scheme into a weak formulation.
\end{Definition}
\begin{Definition}[Weak formulation]
 The finite volume scheme \eqref{scheme_ns_fv} possesses an equivalent formulation  
\begin{subequations}\label{scheme_ns}
\begin{align}
\intO{ D_t \vrh^k \phi_h } &- \sum_{ \sigma \in \facesint } \intSh{  F_h (\vrh^k,\vuh^k)
\jump{\phi_h}   } = 0, \quad \mbox{for all}\ \phi_h \in Q_h,\label{scheme_ns_den}\\
\intO{ D_t  (\vrh^k \vuh^k) \cdot \bfphi_h } &- \sum_{ \sigma \in \facesint } \intSh{ {\bf F}_h  (\vrh^k \vuh^k,\vuh^k)
\cdot \jump{\bfphi_h}   }- \intO{ p_h^k \Divh \bfphi_h } \nonumber \\
&= - 2 \mu  \intO{ \Dhuhk  : \bD_h (\bfphi_h) }
- \lambda  \intO{\Divh   \vuh^k  \; \Divh \bfphi_h },
\quad \mbox{for all } \bfphi_h \in \bQh. \label{scheme_ns_mom} \\
c_v\intO{ D_t (\vrh^k \vth^k) \phi_h } &- c_v\sum_{ \sigma \in \facesint } \intSh{  F_h (\vrh^k\vth^k,\vuh^k)\jump{\phi_h} }
+\intO{ \kappa \Gradedge \vth^k \cdot \Gradedge \phi_h}\nonumber \\
&  = \intO{\left(2\mu \abs{\Dhuhk}^2 + \lambda \abs{\Divh\vuh^k}^2  -p_h^k \Divh \vuh^k \right)\phi_h}, \quad \mbox{for all}\ \phi_h \in Q_h .
\label{scheme_ns_ene}
\end{align}
\end{subequations}
\end{Definition}

%\begin{Lemma}\label{lem_diffusion}
%Any $\vuh\in Q_h$ satisfies the following relation of the diffusion terms
%\begin{equation}\label{inq_diffusion}
%2\intO{|\Dhuh|^2} -\intO{|\Divh \vuh|^2} = \intO{|\Gradh \vuh|^2  } \leq \intO{|\Gradedge \vuh|^2  } .
%\end{equation}
%\end{Lemma}
%\begin{proof}
%Employing \eqref{gradient_inequ} and \eqref{int_by_part_2} we derive \eqref{inq_diffusion}, i.e.,
%\begin{equation*}
%\begin{aligned}
%& 2\intO{|\Dhuh|^2} -  \intO{|\Divh \vuh|^2}
%= 2 \sum_{K \in \grid} |K|  \sumij \biggabs{\frac{\eth_i u_{j,h} + \eth_j u_{i,h} }{2}}^2  - \sum_{K \in \grid} |K| \biggabs{ \sumj \eth_j u_{j,h} }^2
%\\& =
%\sum_{K \in \grid} |K|   \sumij \left( \eth_i u_{j,h}  \eth_j u_{i,h} + \frac{ |\eth_i u_{j,h}|^2 + |\eth_j u_{i,h}|^2 }{2} \right)   -\sum_{K \in \grid} |K| \biggabs{ \sumj \eth_j u_{j,h} \sumi \eth_i u_{i,h}  }
%\\& =
% \sum_{K \in \grid} |K|   \sumij \left( \eth_i u_{j,h}  \eth_j u_{i,h}
%+ |\eth_i u_{j,h}|^2  \right)
%-\sum_{K \in \grid} |K|  \sumj \left( \abs{\eth_j u_{j,h} }^2 +  \sum_{i=1,i\neq j}^d \eth_i u_{i,h} \eth_j u_{j,h} \right)
%%\\&=
%% \sum_{K \in \grid} |K| \sumj \left( \sum_{i=1,i\neq j}^d \eth_i u_{j,h}  \eth_j u_{i,h}
%% + \sum_{i=1,i= j}^d \eth_i u_{j,h}  \eth_j u_{i,h} \right)
%%+ \sum_{K \in \grid} |K| \sumij  |\eth_i u_{j,h}|^2
%%\\&\qquad
%%- \sum_{K \in \grid} |K|  \sumj \left( \abs{\eth_j u_{j,h} }^2 +  \sum_{i=1,i\neq j}^d \eth_i u_{i,h} \eth_j u_{j,h} \right)
%\\& =
%\sum_{K \in \grid} |K|   \sumij  |\eth_i u_{j,h}|^2
%= \intO{|\Gradh \vuh|^2} \leq \intO{|\Gradedge \vuh|^2 }.
%\end{aligned}
%\end{equation*}
%
%\end{proof}
%
\noindent It is suitable to reformulate the convective terms in the following way, see \cite[Lemma 2.5]{FVNS_FLMS}. For reader's convenience we reproduce the proof. 
\begin{Lemma}\label{lem_convective_trans}
For any $r_h , \bfv_h  \in \bQh$, and $\phi \in C^1(\Omega)$, it holds
\begin{align*}
&\intO{ r_h \bfv_h  \cdot \Grad \phi } - \sum_{\sigma \in \facesint} \intSh{ F_h  [ r_h,\bfv_h  ] \ \jump{\Pim \phi }  }
\\&=
 \sum_{\sigma \in \facesint} \intSh{
\left( \frac{1}{2} |\Ov{\bfv_h } \cdot \vc{n}|  +  h^\eps  + \frac{1}{4} \jump{\bfv_h } \cdot \vc{n}   \right) \jump{r_h }  \jump{\Pim  \phi }  }
+ \intO{ r_h \bfv_h  \cdot \left(\Grad \phi - \Gradh \big(\Pim  \phi\big) \right)}.
\end{align*}
\end{Lemma}
\begin{proof}Using the basic equalities \eqref{avg_dif_order}--\eqref{basic_eq2}, we have
\begin{align*}
\intO{ r_h \bfv_h  \cdot \Grad \phi } &= \sum_{K \in \grid} \int_{K} r_h \bfv_h  \cdot \Grad \phi \ \dx
\\&
= \sum_{K \in \grid} \int_{K} r_h \bfv_h  \cdot (\Grad \phi - \Gradh \big(\Pim  \phi\big)) \dx
+ \sum_{K \in \grid} (r_h \bfv_h)_K  \cdot \int_{\partial K}    \vc{n} \, \Ov{\Pim  \phi }  \ds
\\&=
 \intO{ r_h \bfv_h  \cdot (\Grad \phi - \Gradh \big(\Pim  \phi\big) )}
 - \sum_{\sigma \in \facesint} \intSh{ \jump{ r_h \bfv_h  }  \cdot \vc{n} \, \Ov{\Pim  \phi  } }
\\&=
 \intO{ r_h \bfv_h  \cdot (\Grad \phi - \Gradh \big(\Pim  \phi\big) )}
+ \sum_{\sigma \in \facesint} \intSh{ \Ov{ r_h \bfv_h  }  \cdot \vc{n} \jump{ \Pim  \phi  }  }
\\&=
 \intO{ r_h \bfv_h  \cdot (\Grad \phi - \Gradh \big(\Pim  \phi\big) )}
+ \sum_{\sigma \in \facesint} \intSh{ \left( \Ov{ r_h \bfv_h }   - \Ov{ r_h } \ \Ov{\bfv_h } \right) \cdot \vc{n} \jump{ \Pim  \phi  }  }
\\& \qquad
+\sum_{\sigma \in \facesint} \intSh{ \Ov{ r_h } \ \Ov{\bfv_h }  \cdot \vc{n} \jump{\Pim  \phi }  }
\pm
 \sum_{\sigma \in \facesint} \intSh{
\left( \frac{1}{2} |\Ov{\bfv_h } \cdot \vc{n}|  + h^\eps \right) \jump{r_h }  \jump{ \Pim  \phi }  }
\\&=
 \intO{ r_h \bfv_h  \cdot (\Grad \phi - \Gradh \big(\Pim  \phi\big) )}
+\sum_{\sigma \in \facesint} \intSh{ \frac14 \jump{ r_h } \jump{\bfv_h }  \cdot \vc{n} \jump{ \Pim  \phi  }  }
\\& \qquad
+
\sum_{\sigma \in \facesint} \intSh{ F_h  [ r_h,\bfv_h  ]  \jump{\Pim \phi }  } +
\sum_{\sigma \in \facesint} \intSh{\left( \frac{1}{2} |\Ov{\bfv_h } \cdot \vc{n}|  +  h^\eps   \right) \jump{r_h }  \jump{ \Pim  \phi   }} .
\end{align*}
\end{proof}

%Recalling the definition of the diffusive upwind flux \eqref{num_flux}, we have
%\begin{align*}
%&\intO{ r_h \bfv_h  \cdot \Grad \phi } - \sum_{\sigma \in \facesint} \intSh{ Up [ r_h,\bfv_h  ] \ \jump{\Pim \phi }  }
%\\&=
% \sum_{\sigma \in \facesint} \intSh{
%\left( \frac{1}{2} |\Ov{\bfv_h } \cdot \vc{n}|    + \frac{1}{4} \jump{\bfv_h } \cdot \vc{n}   \right) \jump{r_h }  \jump{\Pim  \phi }  }
%+ \intO{ r_h \bfv_h  \cdot \left(\Grad \phi - \Gradh \big(\Pim  \phi\big) \right)}.
%\end{align*}

 Finally, we need a discrete  analogue of the Sobolev--type inequality that can be proved exactly as \cite[Theorem 11.23]{FeiNov_book}.
\begin{Lemma}[Sobolev-type inequality]\label{sobolev_ineq}
Let the function $r \geq 0$ be such that
\[ 0< \intO{r} =c_M, \mbox{ and } \intO{r^\gamma}\leq  c_E \mbox{ for } \gamma>1,
\]
where $c_M$ and $c_E$ are some positive constants.
Then the following Poincar{\'e}--Sobolev type inequality holds true
\begin{equation*}
\norm{v_h}_{L^6(\Omega)} \leq c \norm{\Gradh v_h}_{L^2(\Omega)}^2
+c \left(\intO{r|v_h|}\right)^2
\aleq  c \norm{\Gradh v_h}_{L^2(\Omega)}^2  +c_M
+c \intO{r|v_h|^2}
\end{equation*}
for any $v_h\in Q_h$, where the constant $c$ depends on $c_M$ and $c_E$ but not on the mesh parameter $h$.
\end{Lemma}

\section{Stability}\label{sec:stability_ns}
In this section we show the mass conservation, energy dissipation and entropy inequality for the numerical solutions obtained by the
finite volume scheme \eqref{scheme_ns}.
In what follows we assume  $\vr_h,$ $\vt_h >0.$ 
Note, however, that the non-negativity of the discrete density follows from the renormalized continuity equation Lemma~\ref{lem_renormalized_density} in an analogous way as in \cite{Karper}.

\subsection{Mass conservation}
Setting $\phi_h = 1$ in  \eqref{scheme_ns_den} we derive the mass conservation
\begin{equation}\label{mass_conservation}
\intO{\vrh(t)} = \intO{\vrh(0)} =M_0 >0,\; t\geq 0.
\end{equation}
\subsection{Total energy dissipation}
\begin{Theorem}[Energy balance]\label{thm_energy_stability}
Let $(\vrh,\vuh,\vth)$ satisfy \eqref{scheme_ns}. Then for any $k=1,\ldots,N_T$ it holds
\begin{multline}\label{energy_stability}
 D_t \intO{ \left(\frac{1}{2}  \vrh^k |\vuh^k|^2 + { c_v \vrh^k \vth^k} \right) }
+  h^\eps \sum_{ \sigma \in \facesint } \intSh{  \Ov{ \vrh^k }  \jump{\vuh^k}^2 }
\\
+ \frac{\TS}{2} \intO{ \vrh^{k-1}|D_t \vuh^k|^2  }
+ \frac12 \sum_{ \sigma \in \facesint } \intSh{ (\vrh^k)^{\rm up} |\Ov{\vuh^k} \cdot \vc{n} |\jump{ \vuh^k } ^2   }
= 0
.
\end{multline}
\end{Theorem}
\begin{proof}
We start by recalling the kinetic energy balance, cf.~\cite[equation (3.4)]{FVNS_FLMS},
\begin{multline*} \label{k1}
 D_t \intO{ \frac{1}{2} \vrh |\vuh^k|^2  }
+  2 \mu   \intO{  |\Dhuhk |^2}   + \lambda \intO{  |\Divh \vuh^k|^2}
- \intO{p_h^k \Divh \vuh^k }
\\
+  h^\eps \sum_{ \sigma \in \facesint } \intSh{  \Ov{ \vrh^k }  \jump{\vuh^k}^2 }
+ \frac{\TS}{2} \intO{ \vrh^{k-1}|D_t \vuh^k|^2  }
+ \frac12 \sum_{ \sigma \in \facesint } \intSh{ (\vrh^k)^{\rm up} |\Ov{\vuh^k} \cdot \vc{n} |\jump{ \vuh^k } ^2   } =0.
\end{multline*}
Setting $\phi_h=1$ in \eqref{scheme_ns_ene} we get
\[
D_t \intO{  c_v\vrh^k \vth^k  }
= \intO{\left(2\mu \abs{\Dhuhk}^2 + \lambda \abs{\Divh\vuh^k}^2 - p_h^k \Divh \vuh^k \right)}.
\]
Finally, we sum the previous two equations and finish the proof.
\end{proof}

Theorem \ref{thm_energy_stability}  implies the  energy dissipation
\begin{equation}\label{d_en_ineq_old}
 E_h(t)  \leq E_0,
% \intO{ \left(\frac{1}{2} \vrh |\vuh|^2 +c_v \vr \vt \right) }
%\leq  \intO{ \left(\frac{1}{2} \vrh^0 |\vuh^0|^2 + c_v \vrh^0 \vth^0 \right) } := E^0.
\end{equation}
where
\[   E_h (t)  := \intO{ \left(\frac{1}{2} \vrh(t) |\vuh(t)|^2 +c_v \vrh(t)  \vth(t) \right)}
 \ \mbox{  and  }  \
  E_0 := E_h(0)= \intO{ \left(\frac{1}{2} \vrh^0 |\vuh^0|^2 + c_v \vrh^0 \vth^0 \right) }.
\]
\subsection{First a priori estimates}
Let us summarize {\sl a priori} estimates that we have obtained so far from \eqref{mass_conservation} and \eqref{energy_stability}.
\begin{equation}\label{N}
\vrh \in L^\infty \left(0,T; L^1(\Omega) \right), \quad  \vrh \vuh^2 \in L^\infty \left(0,T; L^1(\Omega) \right), \quad E_h   \in L^\infty \left(0,T; L^1(\Omega) \right), \quad
p_h \in L^\infty \left(0,T; L^1(\Omega) \right) .
\end{equation}
For simplicity, hereafter we denote by $\norm{\cdot}_{L^p} $  and $\norm{\cdot}_{L^p L^q}$ the norms $\norm{\cdot}_{L^p(\Omega)} $  and $\norm{\cdot}_{L^p(0,T; L^q(\Omega))}$, respectively.

\subsection{Entropy equation}
The physical entropy for the perfect gas law is defined  as  a function of density $\vr$ and temperature $\vt$  as
\[s(\vr, \vt)= \log \left(\frac{\vt^{c_v}}{\vr} \right),\]
 and  can be rewritten in terms of density $\vr$ and pressure $p$ as
\[s=s(\vr,p)=\frac{1}{\gamma-1}\log \left(\frac{p}{\vr^\gamma} \right), \ \quad \gamma=\frac{1}{c_v} +1. \]
Then, it is  easy to realize that
$$
(\vr,p)  \mapsto -\vr s(\vr,p)= -\frac{\vr}{\gamma-1} \log \left(\frac{p}{\vr^\gamma}\right)
$$
is a convex function of $(\vr, p)$ for $\vr>0$ and $p>0$.
Moreover, it holds
\begin{equation}\label{rho_s_convex}
\nabla_\vr(-\vr s) = c_v+1 -s, \quad \nabla_p(-\vr s)= - c_v/\vt .
\end{equation}

Before deriving the discrete entropy inequality, we  list two renormalized equations.  We shall use the notation $\co{A}{B} \equiv [ \min\{A,B\} , \max\{A,B\}]$ in what follows. 
\begin{Lemma}\cite[Section 4.1]{FKN2016}(Renormalized continuity equation)\label{lem_renormalized_density}
Let $(\vrh^k,\vuh^k)$ satisfy \eqref{scheme_ns_den}. Then for any $\phi_h \in Q_h$ and any function $B$ that is $C^2$ on the range of $\vr_h^k$
we have
\begin{multline}\label{renormalized_density}
\intO{ D_t B(\vrh^k) \phi_h   }  - \sum_{ \sigma \in \facesint } \intSh{  Up [ B(\vrh^k), \vuh^k ] \jump{ \phi_h } }
 + \intO{ \phi_h\left( B'(\vrh^k) \vrh^k - B(\vrh^k) \right) \Divh \vuh^k    }
\\
 = - \intO{ \frac{ \TS}{2} B''({ \xi^k_{\vr,h}}) | D_t \vrh^k |^2 \phi_h   }
 - \sum_{ \sigma \in \facesint } \intSh{  \frac{B''( \eta^k_{\vr,h})}{2} \jump{ \vrh^k }^2  |\Ov{\vuh^k} \cdot \vn |  \phi_h }
  - h^{\eps}  \sum_{ \sigma \in \facesint } \intSh{ \jump{ \vrh^k } \jump{B'( \vrh^k)\phi_h}  },
\end{multline}
where $\xi^k_{\vr,h} \in \co{\vrh^{k-1}}{\vrh^k}$ and $\eta^k_{\vr,h}\in \co{\vr^k_K}{\vr^k_L}$ for any $ \sigma (= K|L) \in \facesi$,
$i = 1,\ldots,d.$
\end{Lemma}
\begin{Lemma}\cite[Lemma 3.3]{HS_NSF}(Renormalized  internal energy equation)\label{lem_renormalized_energy}
Let $(\vrh, \vuh, \vth)$ satisfy  equation  \eqref{scheme_ns_ene}. Then  for any $\sigma \in K|L$ there exists
$\xi_{\vt,h}^k \in \co{\vth^{k-1}}{\vth^k}$ and $\eta_{\vt,h}^k \in \co{\vt_K^{k}}{\vt_L^k}$, such that for any $\phi_h \in Q_h$, and  any function $\chi$ that is $C^2$ on the range of $\vt_h^k$ it holds
\begin{equation}\label{eq_renormalized_energy}
\begin{split}
c_v &\intO{ D_t \left(\vrh^k \chi(\vth^k) \right) \phi_h  } -c_v \sum_{ \sigma \in \facesint } \intSh{ Up(\vrh^k \chi(\vth^k), \vuh^k) \jump{\phi_h} }
+\sum_{ \sigma \in \facesint } \intSh{ \frac{\kappa}{h} \jump{\vth^k} \jump{\chi'(\vth^k) \phi_h} }
\\  = &
\intO{ \left(2\mu |\Dhuhk|^2  + \lambda |\Divh \vuh^k|^2   -p_h^k \Divh \vuh^k \right)  \chi'(\vth^k) \phi_h}
-\frac{c_v \TS}{2} \intO{ \chi''(\xi_{\vt,h}^k) \vrh^{k-1} |D_t \vth^k|^2 \phi_h }
\\ &
+ \frac{c_v}{2} \sum_{ \sigma \in \facesint } \intSh{\chi''(\eta_{\vt,h}^k)\jump{\vth^k}^2 (\vrh^k)^{\rm out} [ \Ov{\vuh^k} \cdot \vc{n}]^- \phi_h }
- c_v h^\eps \sum_{ \sigma \in \facesint } \intSh{ \jump{\vrh^k} \jump{ \left(\chi(\vth^k) - \chi'(\vth^k)\vth^k \right) \phi_h} }
%%\\ & - c_v h^\eps \sum_{ \sigma \in \facesint } \intSh{ \jump{\vrh^k \vrh^k} \jump{  \chi'(\vth^k) \phi_h} }
\\ &   - c_v h^{\eps}  \sum_{ \sigma \in \facesint } \intSh{ \jump{ \vrh^k\vt_h^k } \jump{\chi'( \vt^k)\phi_h}  },
\end{split}
\end{equation}
\end{Lemma}
Now, we are ready to derive the discrete entropy equation  for the numerical solution of  scheme \eqref{scheme_ns}. 

\begin{Lemma}[Entropy equation]\label{lem_entropy_stability}
Let $(\vrh, \vuh, \vth) $  be the solution of our finite volume scheme \eqref{scheme_ns} such that $\vr_h^k,$ $\vt_h^k>0$ for all $k=1,\ldots,N_T.$ Then,  for any $\phi_h \in Q_h$ it holds
\begin{multline}\label{eq_entropy_stability}
 \intO{ D_t \left(\vrh^k s_h^k \right) \phi_h  } -
\sum_{\sigma \in \facesint } \intSh{ Up(\vrh^k s_h^k, \vuh^k) \jump{\phi_h} }
+ \intO{\kappa \Gradedge\vth^k \cdot \Gradedge\!\!\left(\frac{\phi_h}{\vth^k}\right)}
\\  - \intO{ \left(2\mu |\Dhuhk|^2  + \lambda |\Divh \vuh^k|^2  \right) \frac{\phi_h}{\vth^k} }    =
\intO{\left( D_1 \phi_h   + D_2 \Ov{ \phi_h}  + D_3 \cdot \Gradedge \phi_h \right) },
\end{multline}
where
\begin{equation}\label{entropy_dissipation}
\begin{aligned}
 D_1 & :=  \frac{ \TS}{2 \xi^k_{\vr,h} } | D_t \vr^k_h |^2
+ \frac{h}{2 \eta^k_{\vr,h}}  |\Gradedge \vrh^k|^2    |\Ov{\vuh^k} \cdot \vn |
+ \frac{c_v \TS}{2 |\xi_{\vt,h}^k|^2 }  \vrh^{k-1} |D_t \vth^k|^2
- \frac{c_v h }{2 |\eta_{\vt,h}^k|^2} |\Gradedge\vth^k|^2
(\vrh^k)^{\rm out} [\Ov{\vuh^k} \cdot \vc{n}]^- ,
\\
D_2 &:=
h^{\eps+1}  \Gradedge \vrh^k \cdot  \Gradedge  \left(\nabla_\vr(-\vrh^k s_h^k) \right)
+  h^{\eps+1} \Gradedge p_h^k \cdot \Gradedge \left(\nabla_p(-\vrh^k s_h^k) \right)
,\\
D_3 &:= h^{\eps+1}   \Gradedge \vrh^k  \cdot \Ov{  \nabla_\vr(-\vrh^k s_h^k) }
+ h^{\eps+1} \Gradedge p_h^k  \cdot  \Ov{\nabla_p(-\vrh^k s_h^k)},
\end{aligned}
\end{equation}
and
 $\xi^k_{\vr,h},$ $\eta^k_{\vr,h}$ and $\xi^k_{\vt,h},$ $\eta^k_{\vt,h}$ are given in Lemmas~\ref{lem_renormalized_density} and  \ref{lem_renormalized_energy}, respectively. Moreover, $D_1,$ $D_2 \geq 0.$
\end{Lemma}
\begin{proof}
Firstly, setting $B(\vr)= \vr \log(\vr)$ in the renormalized density equation (\ref{renormalized_density}) implies
\begin{multline}\label{total_ener2}
 \intO{ D_t \left(  \vr^k_h \log(\vrh^k) \right) \phi_h}
 - \sum_{ \sigma \in \facesint } \intSh{  Up [ \vrh^k \log(\vrh^k), \vuh^k ] \jump{ \phi_h } }
 + \intO{  \vrh^k \Divh \vuh^k \phi_h }
\\ =
 - \intO{ \frac{ \TS}{2 \xi^k_{\vr,h} }  | D_t \vr^k_h |^2 \phi_h }
- \sum_{ \sigma \in \facesint } \intSh{ \frac{\phi_h}{2 \eta^k_{\vr,h}} \jump{ \vrh^k }^2    |\Ov{\vuh^k} \cdot \vn | }
- h^\eps \sum_{ \sigma \in \facesint } \intSh{ \jump{ \vrh^k} \jump{\left(\log(\vrh^k)+1\right) \phi_h } }.
\end{multline}
Next, we set  $\chi(\vt)=\log(\vt)$ in  \eqref{eq_renormalized_energy} to get 
% Lemma \ref{lem_renormalized_energy} and derive
\begin{equation}\label{eq_renormalized_energy2}
\begin{aligned}
 c_v & \intO{ D_t \left(\vrh^k \log(\vth^k) \right) \phi_h  } -
c_v\sum_{\sigma \in \facesint } \intSh{ Up (\vrh^k \log(\vth^k), \vuh^k) \jump{\phi_h} }
+\sum_{ \sigma \in \facesint } \intSh{ \frac{\kappa}{h} \jump{\vth^k} \jump{\frac{\phi_h}{\vth^k} } }
\\= &
\intO{ \left(2\mu |\Dhuhk|^2  + \lambda |\Divh \vuh^k|^2   -p(\vrh^k) \Divh \vuh^k \right) \frac{\phi_h}{\vth^k} }
+ \frac{c_v \TS}{2} \intO{  \vrh^{k-1} \biggabs{\frac{D_t \vth^k}{\xi_{\vt,h}^k}}^2 \phi_h }
\\ &
- \frac{c_v}{2} \sum_{ \sigma \in \facesint } \intSh{
\biggabs{\frac{\jump{\vth^k}}{\eta_{\vt,h}^k}}^2
(\vrh^k)^{\rm out} [ \Ov{\vuh^k} \cdot \vc{n}]^- }
- c_v h^\eps \sum_{ \sigma \in \facesint } \intSh{ \jump{\vrh^k} \jump{  \left( \log(\vth^k)-1\right) \phi_h} }
% \\ & - c_v h^\eps \sum_{ \sigma \in \facesint } \intSh{ \jump{\vrh^k \vth^k} \jump{\frac{\phi_h}{\vth^k}} }
\\ &    - c_v h^{\eps}  \sum_{ \sigma \in \facesint } \intSh{ \jump{ \vrh^k\vt_h^k } \jump{\frac{\phi_h}{\vt_h^k}}  }.
\end{aligned}
\end{equation}
 Subtracting \eqref{total_ener2} from \eqref{eq_renormalized_energy2} %keeping in mind \eqref{rho_s_convex}. Indeed, we finally have
yields
\begin{equation*}
\begin{split}
& \intO{ D_t \left(\vrh^k s_h^k \right) \phi_h  } -
\sum_{\sigma \in \facesint } \intSh{ Up(\vrh^k s_h^k, \vuh^k) \jump{\phi_h} }
\\ & \qquad
+\sum_{ \sigma \in \facesint } \intSh{ \frac{\kappa}{h} \jump{\vth^k} \jump{\frac{\phi_h}{\vth^k} } }
- \intO{ \left(2\mu |\Dhuhk|^2  + \lambda |\Divh \vuh^k|^2  \right) \frac{\phi_h}{\vth^k} }
\\& =
 \intO{ \frac{ \TS}{2 \xi^k_{\vr,h} }  | D_t \vr^k_h |^2 \phi_h }
+ \sum_{ \sigma \in \facesint } \intSh{ \frac{\phi_h}{2 \eta^k_{\vr,h}} \jump{ \vrh^k }^2    |\Ov{\vuh^k} \cdot \vn | }
\\ & \qquad
+ \frac{c_v \TS}{2} \intO{  \vrh^{k-1} \biggabs{\frac{D_t \vth^k}{\xi_{\vt,h}^k}}^2 \phi_h }
- \frac{c_v}{2} \sum_{ \sigma \in \facesint } \intSh{
\biggabs{\frac{\jump{\vth^k}}{\eta_{\vt,h}^k}}^2
(\vrh^k)^{\rm out} [ \Ov{\vuh^k} \cdot \vc{n}]^- \phi_h }
\\ &\qquad
+ h^\eps \sum_{ \sigma \in \facesint } \intSh{ \jump{ \vrh^k} \jump{\left(\log(\vrh^k)+1 -c_v\log(\vth^k)+c_v\right) \phi_h } }
  - c_v h^{\eps}  \sum_{ \sigma \in \facesint } \intSh{ \jump{ \vrh^k\vt_h^k } \jump{\frac{\phi_h}{\vt_h^k}}  }.
\end{split}
\end{equation*}
 We finish the derivation of \eqref{eq_entropy_stability} by applying the product rule \eqref{product_rule} on  the last two terms,  rewritten in a convenient way using the identities \eqref{rho_s_convex} and the notation of the discrete operator \eqref{grad_edge}, such that
\begin{align*}
  h^\eps  &\sum_{ \sigma \in \facesint } \intSh{ \jump{ \vrh^k} \jump{ \left( c_v+1-s_h^k \right)\phi_h}} + h^\eps  \sum_{ \sigma \in \facesint } \intSh{ \jump{ \vrh^k \vt_h^k} \jump{\left(-\frac{c_v}{\vt_h^k}\right)\phi_h}} \\
 % & = h^\eps  \sum_{ \sigma \in \facesint } \intSh{ \jump{ \vrh^k} \jump{\nabla_{\vr}(-\vr_h^k s_h^k)}} + h^\eps  \sum_{ \sigma \in \facesint } \intSh{ \jump{ \vrh^k \vt_h^k} \jump{\nabla_p(-\vr_h^k s_h^k)}}
 & =  h^\eps  \sum_{ \sigma \in \facesint } \intSh{ \jump{ \vrh^k} \jump{\nabla_{\vr}(-\vr_h^k s_h^k)\phi_h}} + h^\eps  \sum_{ \sigma \in \facesint } \intSh{ \jump{ \vrh^k \vt_h^k} \jump{\nabla_p(-\vr_h^k s_h^k)\phi_h}} \\
&= h^{\eps+1} \intO{  \Gradedge \vrh^k \cdot  \Gradedge  \left(\nabla_\vr(-\vrh^k s_h^k) \phi_h\right)}
+  h^{\eps+1} \intO{\Gradedge p_h^k \cdot \Gradedge \left(\nabla_p(-\vrh^k s_h^k)\phi_h \right)}\\
& =\intO{\left(   D_2 \Ov{ \phi_h}  + D_3 \cdot \Gradedge \phi_h \right) } .
\end{align*}
The term $D_1$ is obviously non--negative, and by the convexity of the entropy $-\vr s(\vr,p)$ we can conclude that the term $D_2$ is  non--negative as well. Indeed, gradient of any convex sufficiently smooth function is a monotone map.
\end{proof}

\subsection{Discrete entropy inequality}
The discrete entropy inequality is now a direct consequence of Lemma~\ref{lem_entropy_stability}. Indeed,  we set $\phi_h=1$  in the entropy equality \eqref{eq_entropy_stability}  and get
\begin{equation}\label{est_entropy}
\begin{split}
\intO{ D_t \left(\vrh^k s_h^k \right)   }  &=
- \intO{\kappa \Gradedge\vth^k \cdot  \Gradedge\!\!\left(\frac{1}{\vth^k}\right)}
+ \intO{ \frac{1}{\vth^k}\left(2\mu |\Dhuhk|^2  + \lambda |\Divh \vuh^k|^2  \right)  }   +\calDh,
\end{split}
\end{equation}
where $\displaystyle \calDh = \intO{D_1+D_2} \geq 0$ represents the numerical entropy production, cf. \eqref{entropy_dissipation}. 
 The first two terms in \eqref{est_entropy} standing for the discrete counterpart of the physical entropy production are obviously non--negative.
 To exploit some useful estimates from the entropy production,  it is crucial to keep the discrete entropy bounded. To this end we assume the following uniform bounds on the density  and temperature:
\begin{subequations}\label{assumptions}
\begin{equation}\label{rho_assump}
\text{(A1)} \quad 0< \underline{\vr} \leq  \vrh \leq \bar{\vr} \ \ \mbox{ uniformly for all } h \to 0 ,
\end{equation}
\begin{equation}\label{temperature_assump}
\text{(A2)} \quad 0< \underline{\vt} \leq  \vth \leq \bar{\vt} \ \ \mbox{ uniformly for all } h \to 0 .
\end{equation}
\end{subequations}
Clearly,  the  assumptions (A1) and (A2) imply
\begin{equation}\label{entropy_bounds}
 \underline{s} \leq  s_h \leq \bar{s} \ \ \mbox{ uniformly for all } h \to 0.
\end{equation}

\subsection{Second a priori estimates}
 In what follows we   derive  the second {\sl a priori} estimates from the  energy equation and the entropy inequality. 
Firstly,  from the energy equation \eqref{energy_stability}, under the assumptions \eqref{assumptions},  we directly get  the following estimates
\begin{subequations}\label{apriori_2.1}
\begin{equation}\label{est1}
  h^\eps \int_0^T \sum_{ \sigma \in \facesint } \intSh{   \jump{\vuh}^2 } \aleq 1,
\end{equation}
\begin{equation}\label{est2}
  \int_0^T  \sum_{ \sigma \in \facesint } \intSh{  |\Ov{\vuh} \cdot \vc{n} |\jump{ \vuh } ^2   } \aleq 1.
\end{equation}
\end{subequations}
Secondly, the entropy inequality \eqref{est_entropy} and the assumptions \eqref{assumptions} imply
\begin{subequations}\label{apriori_2.2}
\begin{equation}\label{est3}
 - \int_0^T \intO{ \Gradedge\vth \cdot \Gradedge\!\!\left(\frac{1}{\vth}\right)} \dt
\aleq 1 ,
\end{equation}
\begin{equation}\label{est4}
\int_0^T  \intO{ \left(2\mu |\Dhuh|^2  + \lambda |\Divh \vuh|^2  \right)} \dt \aleq 1 ,
\end{equation}
\begin{equation}\label{est6}
\int_0^T \sum_{ \sigma \in \facesint } \intSh{|\Ov{\vuh} \cdot \vn | \jump{  \vr_h}^2     }  \aleq 1,
\end{equation}
 and also
\begin{equation}\label{est5}
\int_0^T\intO{D_1 + D_2\,} \dt \aleq 1 .
\end{equation}

\noindent   Using Lemma \ref{sobolev_ineq} with   \eqref{N},  \eqref{est4} and (A1), we infer that
\begin{equation}\label{ul2l6}
\norm{ \vuh }_{L^2 L^6}   \aleq 1 .
\end{equation}
Further, applying \cite[Lemma 5.1]{FKN2016} with $F(\vt_h)=\vt_h,$ $G(\vt_h)=(\vt_h)^{-1},$ and (A2) we obtain
\begin{align*}
-\frac{\kappa}{h} \sum_{ \sigma \in \facesint } \intSh{  \jump{\vth^k} \jump{\frac{1}{\vth^k} } } \geq \frac{1}{4} \frac{\kappa}{h} \sum_{ \sigma \in \facesint } \intSh{ \left|\frac{\jump{\vt_h}}{\Ov{\vt_h}}\right|^2 }\ageq \frac{\kappa}{h} \sum_{ \sigma \in \facesint } \intSh{ \jump{\vt_h}^2},
\end{align*}
which combined with estimate \eqref{est3} gives the bound on the temperature gradient
\begin{align}\label{grad_vt}
\int_0^T\intO{|\Gradedge\vt_h|^2} \dt = \int_0^T  \sum_{ \sigma \in \facesint } \intSh{ \frac{\jump{\vt_h}^2}{h}} \dt \aleq 1.
\end{align}
\noindent Thanks to the assumptions \eqref{assumptions} we also have
\begin{equation}\label{est_D3}
\begin{aligned}
\int_0^T \intO{|D_3|} \dt  &\aleq
h^{\eps +1} \int_0^T  \sum_{ \sigma \in \facesint } \intSh{  \big|\jump{\vrh} \Ov{(c_v+1- s_h) } \big| } \dt  + h^{\eps +1} \int_0^T  \sum_{ \sigma \in \facesint } \intSh{ \left| \jump{\vr_h\vt_h}\Ov{\left( \frac{-c_v}{\vt_h}\right) }\right|}
\\& \aleq
h^{\eps +1} \int_0^T  \sum_{ \sigma \in \facesint } \intSh{ \abs{\jump{\vrh}}  + \left| \jump{\vr_h\vt_h}\right| } \dt  \aleq h^{\eps +1}  \int_0^T  \sum_{ \sigma \in \facesint } \intSh{  \Ov{\vrh} + \Ov{\vrh\vt_h}  }   \dt\\ &
 \aleq h^{\eps },
\end{aligned}
\end{equation}
where we have used the fact that $ | \jump{r_h} | \leq 2 \Ov{r_h}$  for all $ r_h\geq 0.$
\end{subequations}

\section{Consistency} \label{sec_consistency}
 In this section, our aim is to show the consistency of the discrete continuity and momentum equations \eqref{scheme_ns_den} - \eqref{scheme_ns_mom}, and the discrete entropy equation \eqref{eq_entropy_stability}, i.e. that there exist  $\beta_i>0,$ $i=1, 2, 3,$ such that the numerical solution for  $h \to 0$ satisfies 
\begin{align*}
- \intO{ \vrh^0 \phi(0,\cdot) }  =&
\int_0^\tau \intO{ \left[ \vrh \partial_t \phi + \vrh \vuh \cdot \Grad \phi \right]} \dt  +  \order(h^{\beta_1}), 
\\
- \intO{ \vrh^0 \vuh^0 \bfphi(0,\cdot) }  =&
\int_0^T \intO{ \left[ \vrh \vuh \cdot \partial_t \bfphi + \vrh \vuh \otimes \vuh  : \Grad \bfphi  + p_h \Div \bfphi \right]} \dt,
\\
 - 2  \mu  &  \int_0^T \intO{  \Dhuh : \mathbf D( \bfphi)}  \dt
 -  \lambda \int_0^T \intO{ \Divh \vuh\, \Div \bfphi}\dt  +  \order(h^{\beta_2}),\\
%\end{align*}
%Furthermore, we need to verify the consistency of the discrete entropy inequality
%\begin{multline*}
- \intO{ \vrh^0 s_h^0 \phi(0,\cdot) }  =
&\int_0^T\intO{ [ \vrh s_h \pd_t\phi + \vrh  s_h \vuh \cdot \Grad \phi ] }
-  \int_0^T\intO{ \kappa \Gradedge \vth \cdot \left(\frac{1}{\vth}\Grad \phi +\phi \Gradedge \left(\frac{1}{\vth}\right) \right) }
\\ + &\int_0^T\intO{ \left(2\mu |\Dhuh|^2  + \lambda |\Divh \vuh|^2  \right) \frac{ \phi}{\vth} }
+\int_0^T\intO{(D_1 +D_2)\phi} \dt +  \order(h^{\beta_3}),
\end{align*}
for all test functions $\phi\in C^2([0,T] \times  \Omega),$   $\bfphi  \in C^2([0,T] \times  \Omega; \mathbb R^d)$  with $\phi(T)=0=\bfphi(T)$,   and $D_1 , D_2 \geq 0$ given in Lemma~\ref{lem_entropy_stability}.

%\subsection{Consistency errors}
To this end we proceed  with each term step by step and estimate the consistency errors. We choose the corresponding piecewise constant test functions $\Pim \phi$ and $\Pim \bfphi$ in equations \eqref{scheme_ns_den}, \eqref{eq_entropy_stability} and \eqref{scheme_ns_mom}, respectively.  
For convenience, hereafter we use  $r_h$ for either $\vrh$, $\vrh \uih$ or $\vrh s_h$,  and also $\Pim \phi$ for $\Pim \phi_i,$ $i=1,\ldots,d$.

\subsection{Step 1 -- time derivative terms} The time derivative term can be rewritten as
\begin{equation*}
\begin{aligned}
&\int_0^T \intO{ D_t r_h \Pim \phi } \dt =
\int_0^T \intO{ \frac{r_h(t)- r_h(t-\Delta t) } {\Delta t} \phi(t) } \dt
\\ &
= \frac{1}{\TS}\int_0^T \intO{ r_h(t) \phi(t) } \dt -
\frac{1}{\Delta t}\int_{-\Delta t}^{T-\Delta t} \intO{ r_h(t) \phi(t+ \Delta t) } \dt
\\ &=  - \int_0^T \intO{ r_h(t) D_t \phi(t) } \dt
+ \frac{1}{\Delta t}\int_{T-\Delta t}^{T} \intO{ r_h(t) \phi(t+ \Delta t) } \dt
 - \frac{1}{\Delta t}\int_{-\Delta t}^{0} \intO{ r_h(t) \phi(t+ \Delta t) } \dt
\\ & = - \int_0^T \intO{ r_h(t) \left(\pd_t \phi (t) + \frac{ \TS}{2} \pd_t^2 \phi(t^*) \right)} \dt
- \intO{ r_h^0 \phi (0) }, \text{ for a suitable } t^*.
\end{aligned}
\end{equation*}
Using {\sl a priori} estimates \eqref{N} and \eqref{entropy_bounds}, we derive for $r_h$ being $\vrh$,   $\vrh \uih$ and $\vrh s_h$ that
\begin{equation*}\label{consistency_timederivative}
\begin{aligned}
&\int_0^T \intO{ D_t r_h \Pim \phi } \dt  + \int_0^T \intO{ r_h(t) \pd_t \phi (t) } \dt
+ \intO{ r_h^0 \phi (0) }
\aleq  \TS \norm{r_h}_{L^1L^1} \norm{\phi}_{C^2} \aleq h.
\end{aligned}
\end{equation*}

\subsection{Step 2 -- convective terms}
To deal with the convective terms, it is convenient to recall  the identity from Lemma \ref{lem_convective_trans},
\begin{align*}
\int_0^T \intO{  r_h \vuh \cdot \Grad \phi  } \dt  - \int_0^T \sum_{\sigma \in \facesint} \intSh{ F_h  [ r_h,\vuh ] \jump{  \Pim \phi}}  \dt =\sum_{j=1}^4 E_j(r_h),
\end{align*}
where the error terms can be bounded using the interpolation error estimates \eqref{n4c} and \eqref{n4c2} as follows
\begin{equation*}
\begin{aligned}
& E_1(r_h)= \frac12 \int_0^T \sum_{\sigma \in \facesint} \intSh{ |\Ov{\vuh} \cdot \vc{n}| \jump{r_h }  \jump{ \Pim \phi}  }  \dt
\lesssim   h \norm{\phi}_{C^1}  \int_0^T \sum_{\sigma \in \facesint} \intSh{  |\Ov{\vuh} \cdot \vc{n}|\,| \jump{r_h } | }  \dt
\\& E_2(r_h)= \frac14  \int_0^T \sum_{\sigma \in \facesint} \intSh{ \jump{\vuh} \cdot \vc{n}   \jump{r_h }  \jump{  \Pim \phi}  }  \dt
\lesssim  h   \norm{\phi}_{C^1}  \int_0^T \sum_{\sigma \in \facesint} \intSh{ | \jump{\vuh} \cdot \vc{n}   \jump{r_h } | }  \dt
\\& E_3(r_h)= \int_0^T \intO{  r_h \vuh \cdot \Big(\Grad \phi - \Gradh  \big( \Pim \phi\big) \Big) } \dt
\lesssim  h \norm{\phi}_{C^2} \int_0^T \intO{ | r_h \vuh |} \dt
\\& E_4(r_h)=
%\alpha
 \muh  \int_0^T \sum_{\sigma \in \facesint} \intSh{   \jump{r_h }  \jump{  \Pim \phi}  }  \dt
\lesssim     h^{\eps+1}  \norm{\phi}_{C^1}  \int_0^T \sum_{\sigma \in \facesint} \intSh{ |   \jump{r_h } | }  \dt.
\end{aligned}
\end{equation*}

\subsubsection*{Error terms $E_1(r_h)$}
Firstly, by setting $r_h=\vrh$ in $E_1(r_h)$ we derive
\begin{equation}\label{N3}
\begin{aligned}
 E_1(\vrh) & \aleq  h  \norm{\phi}_{C^1}  \int_0^T \sum_{\sigma \in \facesint} \intSh{  |\Ov{\vuh} \cdot \vc{n}|\,| \jump{\vrh } |}  \dt
\\& \lesssim h  \left(  \int_0^T \sum_{\sigma \in \facesint} \intSh{ |\Ov{\vuh} \cdot \vc{n}  | }  \dt  \right)^{1/2}
 \left( \int_0^T \sum_{\sigma \in \facesint} \intSh{ |\Ov{\vuh} \cdot \vc{n}| \jump{\vrh }^2 }  \dt  \right)^{1/2}
 \\&  \aleq  h^{1/2} \norm{\vuh}_{L^1L^1}^{1/2}   \aleq   h^{1/2} \norm{\vuh}_{L^2L^6}^{1/2} \aleq  h^{1/2}  ,
\end{aligned}
\end{equation}
where we have used the H\"older inequality with the estimates \eqref{est6} and \eqref{ul2l6}.
\\
Secondly, for $r_h=\vrh \uih$ we  control the consistency error  $E_1(r_h)$ as follows
\begin{equation*}
\begin{aligned}
 E_1(\vrh \uih) &\aleq  h   \norm{\phi}_{C^1}  \int_0^T \sum_{\sigma \in \facesint} \intSh{  |\Ov{\vuh} \cdot \vc{n}|\, | \jump{\vrh \uih} | }  \dt
\\&
\lesssim h   \int_0^T \sum_{\sigma \in \facesint} \intSh{  |\Ov{\vuh} \cdot \vc{n}|\, | \jump{\vrh }|\, |  \Ov{\uih} | }  \dt
+ h   \int_0^T \sum_{\sigma \in \facesint} \intSh{  |\Ov{\vuh} \cdot \vc{n}|\, | \jump{ \uih}|\, | \Ov{\vrh}| }  \dt
\\&
\lesssim h  \left(  \int_0^T \sum_{\sigma \in \facesint} \intSh{ |\Ov{\vuh} |^3 }  \dt  \right)^{1/2}
 \left( \int_0^T \sum_{\sigma \in \facesint} \intSh{ |\Ov{\vuh} \cdot \vc{n}| \jump{\vrh }^2 }  \dt  \right)^{1/2}
\\&
+ h  \left(  \int_0^T \sum_{\sigma \in \facesint} \intSh{  |\Ov{\vuh}|^2  }  \dt  \right)^{1/2}
 \left( \int_0^T \sum_{\sigma \in \facesint} \intSh{ \jump{\vuh }^2 }  \dt  \right)^{1/2}
\\ & \aleq   h^{1/2}  +  h^{\frac{1-\eps}{2}} .
 \end{aligned}
\end{equation*}
Here we have used the H\"older inequality, product rule \eqref{product_rule}, the estimates  \eqref{est1}, \eqref{est6},  and the interpolation inequality
\[
\norm{\vuh}_{L^3 L^3} \aleq \norm{\vuh}_{L^\infty L^2} ^{1/2}  \norm{\vuh}_{L^2 L^6}^{1/2}   \aleq 1
\]
with \eqref{N}, \eqref{ul2l6}  and (A1). 
\\
Note that for any $f \in C^1((0,\infty))$  there exists $z_h^* \in \co{z_h^{\rm out}}{z_h^{\rm in}}$
such that the following  estimate holds
\begin{equation} \label{N4}
|\jump{f(z_h)}| =| \nabla f(z_h^*) \jump{z_h} |  \lesssim  | \jump{z_h}|.
\end{equation}
Hence, setting $r_h=\vrh s_h$ in $E_1(r_h)$ and using \eqref{N4} with
 $f(z_h)=\log(z_h),$ $z_h \in \{\vr_h, \vt_h\},$ we finally get
\begin{equation*}
\begin{aligned}
 E_1(\vrh s_h) &\aleq  h  \norm{\phi}_{C^1} \int_0^T \sum_{\sigma \in \facesint} \intSh{  |\Ov{\vuh} \cdot \vc{n}|\, |\jump{\vrh s_h} | }  \dt
\\&
\lesssim h   \int_0^T \sum_{\sigma \in \facesint} \intSh{  |\Ov{\vuh} \cdot \vc{n}|\, | \jump{\vrh } |\,|\Ov{s_h} | }  \dt
+ h   \int_0^T \sum_{\sigma \in \facesint} \intSh{  |\Ov{\vuh} \cdot \vc{n}|\,| \jump{ c_v \log \vth - \log \vrh}|\,| \Ov{\vrh}| }  \dt
\\&
\aleq h   \int_0^T \sum_{\sigma \in \facesint} \intSh{  |\Ov{\vuh} \cdot \vc{n}|\, | \jump{\vrh }  | }  \dt+  h   \int_0^T \sum_{\sigma \in \facesint} \intSh{ |\Ov{\vuh} \cdot \vc{n}|  \left( |\jump{ \vth}|  + |\jump{\vrh}| \right) }  \dt
\\&
 \aleq 2 h   \int_0^T \sum_{\sigma \in \facesint} \intSh{  |\Ov{\vuh} \cdot \vc{n}|\, | \jump{\vrh }  | }  \dt
 + h  \left(  \int_0^T \sum_{\sigma \in \facesint} \intSh{  |\Ov{\vuh}|^2   }  \dt  \right)^{1/2}
 \left( \int_0^T \sum_{\sigma \in \facesint} \intSh{ \jump{  \vth }^2   }  \dt  \right)^{1/2} 
\\&
 \aleq h^{1/2}
 + h \norm{ \vuh}_{L^2L^2}  \aleq  h^{1/2} ,
\end{aligned}
\end{equation*}
where
we have again used  the product rule \eqref{product_rule}, the H\"older inequality,  and the estimates \eqref{N3},    \eqref{est6}, \eqref{grad_vt} and \eqref{N} with the assumptions \eqref{entropy_bounds} and (A1).

\subsubsection*{Error terms $E_2(r_h)$}
To deal with  the second error terms  we first set $r_h = \vrh$ and obtain
\begin{equation*}
\begin{aligned}
 E_2(\vrh) \aleq h \norm{\phi}_{C^1}
  \left(  \int_0^T \sum_{\sigma \in \facesint} \intSh{  \jump{\vuh}^2   }  \dt  \right)^{1/2}
  \left(  \int_0^T \sum_{\sigma \in \facesint} \intSh{  \jump{\vrh}^2   }  \dt  \right)^{1/2}
  \aleq h^{\frac{1-\eps}{2}}
\end{aligned}
\end{equation*}
due to \eqref{est1} and \eqref{rho_assump}.
\\
Then,  inserting $r_h = \vrh \uih$ into $E_2(r_h)$  and  taking into account the estimates  \eqref{N}, \eqref{est1} with (A1) we get in an analogous way as before
\begin{equation*}
\begin{aligned}
 E_2(\vrh \uih) &\aleq h \norm{\phi}_{C^1}
     \int_0^T \sum_{\sigma \in \facesint} \intSh{ \big| \jump{\vuh} \cdot \vc{n} \big(  \jump{\vrh} \Ov{\uih }   + \Ov{\vrh} \jump{\uih }  \big)  \big|}  \dt
\\&
\aleq   h   \int_0^T \sum_{\sigma \in \facesint} \intSh{  |\jump{\vrh}|\, | \jump{\vuh }\cdot \vc{n}| \, | \Ov{\vuh}|   }  \dt
+h   \int_0^T \sum_{\sigma \in \facesint} \intSh{ \Ov{\vrh} \jump{\vuh } ^2  }  \dt
\\&
\aleq h  \left(  \int_0^T \sum_{\sigma \in \facesint} \intSh{  \jump{\vuh}^2 }\dt\right)^{1/2} \left( \int_0^T \sum_{\sigma \in \facesint} \intSh{  | \Ov{\vuh}|^2 }  \dt  \right)^{1/2}
 +h^{ 1-\eps}
 \\&
 \aleq  h^{1-\frac 1 2 -\frac{\eps}{2}}\norm{\vuh}_{L^2L^2} +h^{ 1-\eps} \aleq h^{\frac{1-\eps}{2}} + h^{1-\eps}.
 \end{aligned}
\end{equation*}
Finally,  for $r_h = \vrh s_h$ we deduce, by
  \eqref{N4} with $f(z_h)=\log(z_h),$ $z_h \in \{\vr_h, \vt_h\},$ and \eqref{est1}, \eqref{grad_edge}, (A1),  the bound
\begin{equation*}
\begin{aligned}
 E_2( \vrh s_h) &\aleq  h \norm{\phi}_{C^1}
      \int_0^T \sum_{\sigma \in \facesint} \intSh{ \big| \jump{\vuh} \cdot \vc{n} \big(  \jump{\vrh } \Ov{s_h }   + \Ov{\vrh} \jump{s_h }  \big)  \big|}  \dt
\\& \aleq
 h   \int_0^T \sum_{\sigma \in \facesint} \intSh{ |\jump{\vuh} \cdot \vc{n}|  \big( |\jump{\vrh}| + | \jump{ c_v \log \vth - \log \vrh} \Ov{\vrh}| \big)  }  \dt
\\& \aleq
 h  \left(  \int_0^T \sum_{\sigma \in \facesint} \intSh{ \jump{\vuh} ^2  }  \dt  \right)^{1/2} \Bigg[
 \left( \int_0^T \sum_{\sigma \in \facesint} \intSh{  \jump{\vrh }^2 }  \dt  \right)^{1/2} + \left( \int_0^T \sum_{\sigma \in \facesint} \intSh{ \jump{  \vth }^2   }  \dt  \right)^{1/2}
\Bigg]
\\& \aleq  h^{\frac{1-\eps}{2}}  + h^{\frac{3-\eps}{2}}.
\end{aligned}
\end{equation*}

\subsubsection*{Error terms $E_3(r_h)$}
 The estimates of the third error terms are straightforward due to  \eqref{N}, \eqref{assumptions} and \eqref{entropy_bounds}. Indeed,
\begin{equation*}
\begin{aligned}
 E_3(\vrh) & \aleq  h \norm{\phi}_{C^2} \int_0^T \intO{ | \vrh \vuh |} \dt \aleq h \norm{\vuh}_{L^2L^2} \aleq h ,
\\ E_3(\vrh \uih) & \aleq h \norm{\phi}_{C^2} \int_0^T \intO{ | \vrh \uih \vuh |} \dt \aleq h \norm{\vrh |\vuh|^2}_{L^\infty L^1} \aleq h ,
\\ E_3(\vrh s_h) & \aleq  h \norm{\phi}_{C^2} \int_0^T \intO{ | \vrh s_h \vuh |} \dt  \aleq h \norm{\vuh}_{L^2L^2} \aleq h.
\end{aligned}
\end{equation*}
\subsubsection*{Error terms $E_4(r_h)$}
Finally, we treat  the fourth error terms.
For $r_h =\vrh$ the argumentation is simple and analogous as above. For $r_h=\vrh s_h$ the term is not present, i.e. $E_4(\vrh s_h)=0$. Thus we only concentrate on a slightly more involved estimate for $r_h=\vrh \uih$,
\begin{equation*}
\begin{aligned}
 E_4(\vrh \uih) & \aleq     h^{\eps+1}  \norm{\phi}_{C^1}  \int_0^T \sum_{\sigma \in \facesint} \intSh{ |   \jump{\vrh \uih } | }  \dt
\aleq    h^{\eps+1}    \int_0^T \sum_{\sigma \in \facesint} \intSh{ |   \jump{\vrh  }   \Ov{\uih}|   +  |  \Ov{\vrh} \jump{ \uih } | }  \dt
\\& \aleq
 h^{\eps+1}  \left(  \int_0^T \sum_{\sigma \in \facesint} \intSh{   |\Ov{\vuh}|^2  }  \dt  \right)^{1/2}
\left(  \int_0^T \sum_{\sigma \in \facesint} \intSh{   \jump{ \vrh }^2  }  \dt  \right)^{1/2}
\\
& +
h^{\eps+1}   \left(  \int_0^T \sum_{\sigma \in \facesint} \intSh{   \Ov{\vrh}^2  }  \dt  \right)^{1/2}
\left(  \int_0^T \sum_{\sigma \in \facesint} \intSh{   \jump{ \vuh }^2 }  \dt  \right)^{1/2}
\\ &\aleq  h^{\eps}    +h^{(\eps+1)/2}   ,
\end{aligned}
\end{equation*}
where we have used the assumption (A1), and bounds \eqref{N},  \eqref{est1}.

\medskip
\noindent
 Collecting the above estimates of $E_i(r_h), i=1,\ldots, 4$  for $r_h \in \{\vrh,\vrh\uih,\vrh s_h\},$  we know that there exists a positive $\beta>0$ such that
\[
\int_0^T \intO{  r_h \vuh \cdot \Grad \phi  } \dt  - \int_0^T \sum_{\sigma \in \facesint} \intSh{ F_h  [ r_h,\vuh ] \jump{  \Pim \phi}}  \dt =\sum_{j=1}^4 E_j(r_h) \aleq h^\beta,
\]
provided $ \eps\in(0,1)$.

\subsection{Step 3 -- $\kappa$--term in the entropy equation \eqref{eq_entropy_stability} } 
 Using the product rule \eqref{product_rule} we can write
\[
\begin{aligned}
& \int_0^T \intO{ \kappa \Gradedge \vth \cdot  \left(\frac{1}{\vth}\Grad \phi + \phi \Gradedge  \left(\frac{1}{\vth}\right) \right) } \dt
-\int_0^T \intO{  \kappa \Gradedge{\vth}  \cdot   \Gradedge \left(\frac{\Pim \phi}{\vth}  \right)} \dt
\\&=
\int_0^T \intO{ \kappa \Gradedge \vth \cdot  \left(\frac{1}{\vth}\Grad \phi + \phi \Gradedge  \left(\frac{1}{\vth}\right) \right) } \dt
\\&
-\int_0^T \intO{  \kappa \Gradedge{\vth}  \cdot  \left(  \left( \Gradedge (\Pim \phi) \right) \Ov{\left(\frac{1}{\vth} \right) }  +  \left( \Gradedge \left(\frac{1}{\vth}\right) \right) \Ov{\Pim \phi} \right)
 } \dt
\\&=
\int_0^T \intO{ \kappa \Gradedge\vth \cdot \left(\frac{1 }{\vth}   \Grad\phi - \Ov{\left(\frac{1 }{\vth} \right) }(\Gradedge(\Pim \phi))  \right) } \dt
+ \int_0^T  \intO{ \kappa \Gradedge\vth \cdot \left( \Gradedge \left( \frac{1}{\vth} \right) \right) \left(\phi  - \Ov{ \Pim \phi}   \right) } \dt
\,\,\\
& =:  I_1 +I_2,
\end{aligned}
\]
where the residual terms $I_1$ and $I_2$ shall be controlled in what follows. Applying the H\"older inequality, interpolation estimates \eqref{n4c}, \eqref{n4c2} and \eqref{grad_vt} yields
\[
\begin{aligned}
I_1 &= \int_0^T \intO{ \kappa \Gradedge\vth \cdot \left(\frac{1 }{\vth}   \Grad\phi
- \Ov{\left(\frac{1 }{\vth} \right) }\Gradedge(\Pim \phi)  \pm \Ov{\left(\frac{1 }{\vth} \right) }\Grad \phi  \right) }  \dt
\\& =
\int_0^T \intO{ \kappa \Gradedge\vth \cdot  \Grad\phi \left(\frac{1 }{\vth}
- \Ov{\left(\frac{1 }{\vth} \right) }  \right) }  \dt
+
\int_0^T \intO{ \kappa  \Ov{\left(\frac{1 }{\vth} \right) } \Gradedge\vth \cdot \left( \Grad \phi- \Gradedge(\Pim \phi )  \right) }  \dt
\\&
\aleq h \norm{\Gradedge \vth}_{L^2L^2}  \norm{\Gradedge \left( \frac{1}{\vth} \right) }_{L^2L^2} \norm{\phi}_{C^1} +
h \norm{\Gradedge \vth}_{L^2L^2} \norm{\phi}_{C^2}
\aleq h.
\end{aligned}
\]
Recalling \eqref{N4} with $f(\vt_h)=\log(\vt_h)$ we could infer from \eqref{grad_vt} the bound
\begin{align*}\label{grad_inv_vt}
\norm{\Gradedge \left( \frac{1}{\vth} \right) }_{L^2L^2}  \lesssim \norm{\Gradedge \vth}_{L^2L^2} \lesssim 1.
\end{align*}
By an analogous argument we have
\[
\begin{aligned}
I_2 = \int_0^T  \intO{ \kappa \Gradedge\vth \cdot \left(\Gradedge \left( \frac{1}{\vth} \right) \right) \left(\phi  - \Ov{ \Pim \phi}   \right) } \dt
\aleq
 h \norm{\Gradedge \vth}_{L^2L^2} \norm{\Gradedge \left( \frac{1}{\vth} \right) }_{L^2L^2} \norm{\phi}_{C^1}
\aleq h.
\end{aligned}
\]
Thus we have shown the consistency of  the $\kappa$--term
\begin{equation*}\label{consistency_kappa}
\int_0^T \intO{ \kappa \Gradedge \vth \cdot  \left(\frac{1}{\vth}\Grad \phi + \phi \Gradedge \left( \frac{1}{\vth} \right) \right) } \dt
-\int_0^T \intO{  \kappa  \Gradedge{\vth} \cdot  \left( \Gradedge \left( \frac{\Pim \phi}{\vth} \right) \right)  } \dt \aleq h.
\end{equation*}

\subsection{Step 4 -- dissipation terms}\hspace{1ex} \\
Applying the estimate \eqref{est4}  with (A2) for the dissipation terms in the entropy equation \eqref{eq_entropy_stability} we immediately get
\begin{align*}\label{consistency_dissipation1}
&\int_0^T\intO{ \left(2\mu |\Dhuh|^2  + \lambda |\Divh \vuh|^2  \right) \frac{ \Pim\phi}{\vth} }
-\int_0^T\intO{ \left(2\mu |\Dhuh|^2  + \lambda |\Divh \vuh|^2  \right) \frac{\phi}{\vth} }
\\
&\aleq  h \norm{\phi}_{C^1} \int_0^T\intO{ \left(2\mu |\Dhuh|^2  + \lambda |\Divh \vuh|^2  \right) \frac{ 1}{\vth} }
 \aleq h .
\end{align*}

\subsection{Step 5 -- viscosity terms}
 The interpolation error estimate (\ref{n4c2}) and the {\sl a priori} bound \eqref{est4} are enough to control  the viscosity terms in the momentum equation. Indeed,   we have
\begin{subequations}\label{consistency_viscosity}
\begin{equation*}
\begin{split}
& \int_0^T \intO{ \Dhuh : \bD( \bfphi) }\dt  - \int_0^T \intO{ \Dhuh : \bD_h(\Pim \bfphi) }\dt
\\&=
\int_0^T \intO{ \Dhuh : \big( \bD( \bfphi) - \bD_h(\Pim \bfphi)  \big) }\dt
\lesssim
\norm{\Dhuh }_{L^2L^2} h \norm{\bfphi}_{C^2}
\aleq h,
\end{split}
\end{equation*}
 and analogously, for the divergence term 
\begin{equation*}
\begin{split}
& \int_0^T \intO{ \Divh \vuh\, \Divh (\Pim \bfphi )  } -\int_0^T \intO{\Divh \vuh \, \Div \bfphi}\dt
\\&=
\int_0^T \intO{\Divh \vuh \, \big(\Divh \left(\Pim \bfphi \right)   -\Div \bfphi \big) }\dt
\lesssim
\norm{\Divh \vuh}_{L^2L^2} h \norm{\bfphi}_{C^2}
\lesssim h.
\end{split}
\end{equation*}
\end{subequations}

\subsection{Step 6 -- pressure term}
The pressure term  in the momentum equation is controlled, thanks to the interpolation estimate \eqref{n4c2} and the {\sl a priori} estimate \eqref{N} for the pressure, as 
\begin{equation*}\label{consistency_pressure}
 \int_0^T  \intO{ p_h  \Divh ( \Pim  \bfphi)     } \dt  -  \int_0^T  \intO{ p_h   \Div \bfphi  } \dt
\aleq
\norm{p_h}_{L^\infty L^1} h \norm{\bfphi}_{C^2}
\aleq h.
\end{equation*}

\subsection{Step 7 --  entropy production terms $D_1$, $D_2$ and $D_{3}$ }
 In an analogous way we bound the three entropy production terms in the entropy equation,
\begin{subequations}\label{consistency_dissipation2}
\begin{equation*}
\int_0^T \intO{  (D_1\Pim\phi + D_2  \Ov{\Pim\phi} )} - \int_0^T \intO{ (D_1 +D_2) \phi}  \aleq h  \norm{\phi}_{C^1} \big(\norm{D_1}_{L^1L^1}  +\norm{D_2}_{L^1L^1}\big)  \aleq h ,
\end{equation*}
and
\begin{equation*}
 \int_0^T  \intO{D_{3} \cdot \Gradedge(\Pim\phi)}
\aleq \norm{D_{3}}_{L^1L^1}  \norm{\phi}_{C^2}
\aleq h^\eps,
\end{equation*}
using the {\sl a priori} estimates
 \eqref{est5} and \eqref{est_D3}, respectively.
\end{subequations}

\medskip

Let us summarize the above  calculations leading to the desired consistency formulation of the numerical approximation of the continuity and momentum equations as well as the discrete entropy equation. 
\begin{Lemma}[Consistency of  the continuity and momentum equations] \label{lem_consistency}
Let $(\vrh, \vuh, \vth)$, $h \in (0,h_0)$, $h_0 \ll 1$ be the numerical solution obtained by our finite volume scheme \eqref{scheme_ns} 
with  $\TS\approx h$ and  $0<\eps <1$.
Then there exists $\beta >0$ such that
\begin{equation} \label{cP1}
- \intO{ \vrh^0 \phi(0,\cdot) }  =
\int_0^T \intO{ \left[ \vrh \partial_t \phi + \vrh \vuh \cdot \Grad \phi \right]} \dt  +  \order( h^{\beta}), \;
\end{equation}
for any $\phi \in C^2([0,T] \times  \Omega), \; \phi(T)=0$;
\begin{multline} \label{cP2}
- \intO{ \vrh^0 \vuh^0 \bfphi(0,\cdot) }  =
\int_0^T \intO{ \left[ \vrh \vuh \cdot \partial_t \bfphi + \vrh \vuh \otimes \vuh  : \Grad \bfphi  + p_h \Div \bfphi \right]} \dt
\\
 -  \mu \int_0^T \intO{  \Dhuh  :  \bD( \bfphi) }  \dt
-   \lambda\int_0^T \intO{ \Divh \vuh\, \Div \bfphi}\dt
+ \order( h^{\beta}), \;
\end{multline}
for any $\bfphi \in C^2([0,T] \times { \Omega}; \mathbb R^d), \; \bfphi(T)=0$.
\end{Lemma}

%Our numerical scheme also satisfies entropy inequality.
\begin{Lemma}[Consistency of  the entropy equation]
 Let $(\vrh, \vuh, \vth)$, $h \in (0,h_0)$, $h_0 \ll 1$ be the numerical solution obtained by our finite volume scheme \eqref{scheme_ns} 
with $\TS\approx h$ and  $0<\eps <1$.
Then there exists $\beta >0$ such that
for any $\phi \in C^2([0,T] \times \Omega)$,   $\phi(T)=0,$  it holds that
\begin{multline}\label{consistency_entropy}
- \intO{ \vrh^0 s_h^0 \phi(0,\cdot) }  =
\int_0^T\intO{ [ \vrh s_h \pd_t\phi + \vrh  s_h \vuh \cdot \Grad \phi ] }
-  \int_0^T\intO{ \kappa \Gradedge \vth \cdot \left(\frac{1}{\vth}\Grad \phi +\phi \Gradedge \left(\frac{1}{\vth}\right) \right) }
\\ + \int_0^T\intO{ \left(2\mu |\Dhuh|^2  + \lambda |\Divh \vuh|^2  \right) \frac{ \phi}{\vth} }
+\int_0^T\intO{(D_1  + D_2) \phi} \dt + \mathcal{O} (h^\beta),
\end{multline}
where $D_1+D_2 \in L^1((0,T)\times \Omega)$ are the non--negative numerical entropy production terms, cf. \eqref{entropy_dissipation}.
\end{Lemma}

It should be pointed out here again that the numerical scheme \eqref{scheme_ns} is energy dissipative, cf.~\eqref{d_en_ineq_old}, which means
\begin{equation}\label{d_en_ineq}
 \intO{ \left(\frac{1}{2} \vrh |\vuh|^2 +c_v \vrh \vth \right) }
\leq  \intO{ \left(\frac{1}{2} \vrh^0 |\vuh^0|^2 + c_v \vrh^0 \vth^0 \right) }.
\end{equation}

\section{Convergence of the finite volume method}
\label{sec_convergence}

The aim of this section is to show
%the main result which is
the convergence of our finite volume method \eqref{scheme_ns_fv} to the strong solution of the Navier--Stokes--Fourier system \eqref{ns_eqs}  on the lifespan of the latter.  We begin with the definition of the  DMV solution of \eqref{ns_eqs}  which plays an essential role in the proof of the main result, see also \cite[Definition 2.3]{stability_nsf}.
\begin{Definition}[DMV solution]\label{def_DMV}
A parametrized family of probability measures $\{\Nu\}_{(t,x)\in (0,T)\times\Omega}$ is
%Then a Young measure $\{\Nu\}_{(t,x) \in (0,T)\times\Omega}$ generated by $\{ \vr_\delta, \vt_\delta, \vu_\delta, \mathbf{D}(\vu_\delta),  \Grad \vt_\delta \}_{\delta \searrow 0}$  is called
the \emph{dissipative measure--valued} solution to the Navier--Stokes--Fourier system \eqref{ns_eqs} with the initial condition $\{\mathcal V_{0,x}\}_{x\in\Omega}$ if the following hold:
\begin{itemize}
\item 
the mapping
\[
\mathcal{V}_{t,x}:
(t,x) \in (0,T) \times  \Omega \mapsto \mathcal{P}(\mathcal{F}) \quad \mbox{is weakly-(*) measurable,} \quad
\]
with $\mathcal{P}$ being the space of probability measures defined on the phase space
\[
\mathcal{F} = \left\{ \vr, \vt, \vu, \mathbf D_u, \mathbf D_\vt \ \Big| \ \vr \geq 0,\ \vt \geq 0,\ \vu \in \mathbb R^d,\ \mathbf D_u \in \mathbb R^{d\times d}_{\rm sym},\ \mathbf D_\vt \in \mathbb R^d \right\};
\]
 \item  $\{\Nu\}_{(t,x)\in (0,T)\times\Omega}$ complies with the compatibility condition
 \begin{equation}\label{perpartes}
\begin{aligned}
-\intTO{\langle \Nu; \vu \rangle \cdot \Div \mathbb T} &= \intTO{\langle \Nu; \mathbf D_u \rangle : \mathbb T}, \ &&\mbox{for any } \mathbb T \in C^1([0,T]\times\Omega; \mathbb R^{d\times d}_{\rm sym}) \\
-\intTO{\langle \Nu; \vt \rangle  \Div \bfvarphi } &= \intTO{\langle \Nu; \mathbf D_\vt \rangle  \bfvarphi}, \ &&\mbox{for any } \bfvarphi \in C^1([0,T]\times\Omega;\mathbb R^{d});
\end{aligned}
\end{equation}
\item conservation of mass
\begin{equation} \label{G5}
\left[ \intO{ \left< \Nu ; \vr \right> \varphi (t,x) } \right]_{t = 0}^{t = \tau} =
\int_0^\tau \intO{ \left[ \left< \Nu ; \vr \right> \partial_t \varphi(t,x) + \left< \Nu; \vr \vu \right> \cdot \Grad \varphi(t,x) \right] } \dt
\end{equation}
\hfill for a.a. $\tau \in [0,T]$ and any $\varphi \in C^1( [0,T] \times \Omega)$;

\item balance of momentum
\begin{equation} \label{G6}
\begin{split}
&\left[ \intO{ \left< \Nu ; \vr \vu \right> \cdot \bfvarphi (t,x) } \right]_{t = 0}^{t = \tau} \\ &=
\int_0^\tau \intO{ \left[ \left< \Nu ; \vr \vu \right> \cdot \partial_t \bfvarphi(t,x) + \left< \Nu; \vr \vu \otimes \vu \right> : \Grad \bfvarphi(t,x)
+ \left< \Nu ; p(\vr, \vt) \right> \Div \bfvarphi (t,x) \right] } \dt\\
&+ \int_0^\tau \intO{ \left< \Nu; \mathbb S(\mathbf D_u) \right> : \Grad \bfvarphi(t,x) } \dt +
\int_0^\tau \int_{\Omega} \Grad \bfvarphi : {\rm d}{\nu_C}
\end{split}
\end{equation}
\hfill for a.a. $\tau \in [0,T]$ and any $\bfvarphi \in C^1([0,T] \times \Omega; \mathbb R^d)$,
where $\nu_C \in \mathcal{M}( [0,T] \times  \Omega;\mathbb R^{d \times d})$\footnote{The symbol $\nu_C$ stands for a tensor--valued signed Borel measure and the term $\int_0^\tau \int_{ \Omega} \Grad \bfvarphi : {\rm d}{\nu_C}$ is understood as the value of the functional $\nu_C$ over the continuous function $\Grad \bfvarphi$.}  is  called  \emph{concentration defect measure};

\item energy inequality
\begin{equation} \label{G7}
\intO{ \left< \mathcal{V}_{\tau,x} ; \frac{1}{2} \vr |\vu|^2 + \vr e(\vr, \vt) \right>  }
\leq  \intO{ \left< \mathcal{V}_{0,x} ; \frac{1}{2} \vr |\vu|^2 + \vr e(\vr, \vt) \right>  }
\end{equation}
\hfill for a.a. $\tau \in [0,T]$;

\item entropy inequality
\begin{equation} \label{G10}
\begin{split}
&\left[ \intO{ \left< \Nu; \vr s(\vr, \vt) \right> \varphi(t,x) } \right]_{t = 0}^{t = \tau} \\& \geq
\int_0^\tau \intO{ \left[ \left< \Nu; \vr s(\vr, \vt) \right> \partial_t \varphi(t,x) +
\left< \Nu; \vr s(\vr, \vt) \vu - \frac{\kappa\Grad\vt }{\vt} \right>
\cdot \Grad \varphi(t,x) \right] } \dt\\
&+ \int_0^\tau \intO{ \left< \Nu ; \frac{1}{\vt} \left( \mathbb S(\mathbf D_u) : \mathbf D_u + \frac{\kappa  |\mathbf D_\vt|^2 }{\vt}\right) \right>
\varphi (t,x) } \dt
\end{split}
\end{equation}
\hfill for a.a. $\tau \in [0,T]$ and any $\varphi \in C^1([0,T] \times  \Omega)$, $\varphi \geq 0$;
\item 
The \emph{dissipation defect}   given by
\[
\mathcal{D}(\tau) = \intO{ \left< \mathcal{V}_{0,x}; \frac{1}{2} \vr |\vu|^2 + \vr e(\vr, \vt) \right> } -
\intO{ \left< \mathcal{V}_{\tau,x}; \frac{1}{2} \vr |\vu|^2 + \vr e(\vr, \vt) \right> } \geq 0
\]
and the concentration defect measure $\nu_C$ from \eqref{G6} satisfy
 \begin{equation} \label{G11}
\int_0^T  \psi(t) \int_{ \Omega} {\rm d}|\nu_C|   \aleq \int_0^T \psi(t) \mathcal{D}(t) \ \dt
\end{equation}
\hfill for any $\psi \in C([0,T]), \ \psi \geq 0$.
%\\ and if the \emph{concentration defect measure} $\nu_C \in \mathcal{M}( \Ov{Q}_T ; R^{N \times N})$\footnote{The symbol $\nu_C$ stands for a tensor valued signed Borel measure and the term $\int_0^\tau \int_{\Ov{\Omega}} \Grad \bfphi : {\rm d}{\nu_C}$ is understood as the value of the functional $\nu_C$ over the continuous function $\Grad \bfphi$.} is controlled by the {\it energy dissipation defect}
\\
%The total energy balance (\ref{G7}) implies for the that $ \mathcal{D}\in L^\infty(0,T),$ $\mathcal{D}\geq 0$ a.a. in $(0,T).$
\end{itemize}
%\subsubsection{Concentration defect}\label{CDC}
\qed
\end{Definition}

We refer the reader to, e.g., \cite{BALL2,PED1} for more details on the Young measure.

\begin{Remark}
It should be noted that in \cite[Definition 2.3]{stability_nsf} an additional compatibility condition of Korn--Poincar\' e--type inequality, cf. \cite[(2.13)]{stability_nsf}, was required for the case of no slip/no flux boundary conditions for the velocity and the heat flux, respectively. This condition was needed in the proof of the DMV--strong uniqueness principle, cf.~\cite[Sections 4.1, 5.1.2]{stability_nsf}. In the case of space--periodic boundary conditions the Korn--Poincar\' e inequality does not hold. Nevertheless, the DMV--strong uniqueness principle can be obtained in an analogous way as in \cite{stability_nsf} provided the density is bounded from below.
\end{Remark}

\subsection{Convergence to  a dissipative measure-valued solution}
In view of the assumptions \eqref{assumptions} and  {\sl a priori} estimates \eqref{N},  \eqref{apriori_2.1} and  \eqref{apriori_2.2} 
we may deduce, at least for a subsequence, that the numerical solutions $\{\mathbf U_h\}_{h>0}=\{(\vr_h,\vu_h,\vt_h,\Dhuh,\Gradedge\vth)\}_{h>0}$ in the limit for $h \to 0$  generate  a Young measure $\{ \mathcal{V}_{t,x} \}_{(t,x) \in (0,T) \times \Omega}$, whose
 support  is contained in the set
\[
{\rm supp} [ \mathcal{V}_{t,x} ] \subset
\left\{ \vr, \vt, \vu, \mathbf D_u, \mathbf D_{\vt} \ \Big| \ 0 < \underline{\vr} \leq \vr \leq \Ov{\vr},\ 0 < \underline{\vt} \leq \vt \leq \Ov{\vt}, \ \vu \in \mathbb R^d,\ \mathbf D_u \in \mathbb R^{d\times d}_{\rm sym},\  \mathbf D_\vt \in \mathbb R^d \right\}
\]
for a.a. $(t,x) \in (0,T) \times \Omega.$ 
More specifically, 
\begin{itemize}
\item the mapping
$
\mathcal{V}_{t,x}:
(t,x) \in (0,T) \times  \Omega \mapsto \mathcal{P}(\mathcal{F})
$
is weakly-(*) measurable
\item
$
G(\vU_h ) \to \av{G (\vU)} \ \mbox{weakly-(*) in}\ L^\infty((0,T) \times \Omega)
$
 and
\[
\av{G (\vU)} (t,x) = \int_{\mathcal{F}} G(\vU) {\rm d}\mathcal{V}_{t,x} \equiv
\left< \mathcal{V}_{t,x}; G(\vU) \right> \ \mbox{for a.a.} \ (t,x) \in (0,T) \times \Omega,
\]
for any $G \in C_c( \mathcal{F} )$, $\vU=\left( \vr, \vt, \vu, \mathbf D_u, \mathbf D_\vt \right) \in \mathcal F.$
\end{itemize}
This, in particular, means that all nonlinearities appearing in the consistency formulation \eqref{cP1} -- \eqref{cP2} are weakly precompact in the
Lebesgue space $L^1((0,T) \times \Omega),$ and hence passing to the limit with $h \rightarrow 0$ yields \eqref{G5} -- \eqref{G6} and $\nu_C\equiv 0.$ The compatibility condition \eqref{perpartes} is a direct consequence of \eqref{apriori_2.2}, since
\begin{align*}
  \Dhuh &\to \Du \ \mbox{weakly in} \ L^2((0,T) \times \Omega;
\mathbb R^{d \times d}), \quad
 \Gradedge\vth \to \nabla _x \vt \ \mbox{weakly in} \ L^2((0,T) \times \Omega;
\mathbb R^d ).
\end{align*}
 Now we recall \cite[Lemma~2.1]{FGSWW}  which shall be to used to pass to the limit in the entropy equality.  
\begin{Lemma}
\label{lemma}
Let
\begin{align*}
|F(\vU)|\leq G(\vU) \ \mbox{ for all }\ \vU \in \mathcal{F}.
\end{align*}
Then
\begin{align*}
\left| \av{F(\vU)} - \left< \mathcal{V}_{t,x}; F(\vU) \right> \right| \leq \av{G(\vU)} -  \left< \mathcal{V}_{t,x}; G(\vU) \right>   \mbox{ in } \mathcal{M}([0,T]\times\Omega).
\end{align*}
\end{Lemma}
\noindent
We consider the limit in the entropy equation \eqref{consistency_entropy}. For the nonlinear discrete entropy production terms 
\begin{align*}
\mathcal R_h +  \calP_h:=
\frac{1}{\vt_h}\big(2\mu|\Dhuh|^2+\lambda|\Divh\vuh|^2\big) + \frac{\kappa|\nabla_h\vth|^2}{\vt^2_h} +   \calP_h \geq 0,
\end{align*}
where $\calP_h :=D_1 +D_2 \geq 0$,  cf. \eqref{est_entropy},  we can only assert that
\begin{align*}
&\mathcal R_h  + \calP_h \rightarrow \av{\mathcal R +   \calP}   \ \mbox{weakly-(*) in}\ \mathcal{M}([0,T] \times \Omega; \mathbb R).
\end{align*}
We  apply Lemma~\ref{lemma} for $F(\vU)\equiv 0$  and $$\displaystyle G(\vU)=\frac{1}{\vt}\big(2\mu|\mathbf D_u|^2+\lambda|\mathrm{tr}\,\mathbf D_u|^2\big) + \frac{\kappa|\mathbf D_\vt|^2}{\vt^2} = \frac{1}{\vt}\left(S(\mathbf D_u) :\mathbf D_u + \frac{\kappa|\mathbf D_\vt|^2}{\vt} \right)$$  to get
\begin{align*}
0 \leq \av{\mathcal R }-\left< \mathcal{V}_{t,x};\mathcal R   \right>.
\end{align*}
 Consequently, passing to the limit in the entropy equation \eqref{consistency_entropy} with non-negative test function we derived the entropy inequality \eqref{G10}. 
Similarly,  passing to the limit in the discrete energy inequality \eqref{d_en_ineq} directly yields \eqref{G7}. 
%
% 
%  and  the  entropy production satisfies 
%\begin{align*}\label{sigma}
%\varsigma := \av{\mathcal R} \geq \left< \mathcal{V}_{t,x};\mathcal R  \right>= \left< \mathcal{V}_{t,x};\frac{1}{\vt}\left(\mathbb S(\mathbf D_u) :\mathbf D_u + \frac{\kappa|\mathbf D_\vt|^2}{\vt} \right)\right>.
%\end{align*}
\noindent
%Due to Theorem~\ref{thm_energy_stability} we also have that the dissipation defect $\mathcal D \geq 0.$
Note that  the inequality \eqref{G11} is satisfied since $\nu_c\equiv 0.$ 
Summing up the preceding discussion, we can state the following result.
\begin{Theorem}[Convergence to DMV solution]\label{thm_dmvs}
Let the initial data satisfy the assumptions
% \eqref{assumptions}, i.e.,
\begin{align*}
0< \underline{\vr} \leq \vr_{0,h} \leq \overline{\vr}, \
0< \underline{\vt} \leq \vt_{0,h} \leq \overline{\vt}, \ \|\vu_{0,h}\|_{L^2} \leq \overline{u},
\end{align*}
for some positive constants $\underline{\vr},$ $\overline{\vr},$  $\underline{\vt},$ $\overline{\vt},$ $\overline{u}.$
Let $(\vr_h, \vt_h, \vu_h)$ be the solution of the finite volume scheme \eqref{scheme_ns} with $0 < \eps <1, $ such that the assumptions \eqref{assumptions} hold, i.e.,
\begin{align*}
0< \underline{\vr} \leq \vr_{h}(t) \leq \overline{\vr}, \
0< \underline{\vt} \leq \vt_{h}(t) \leq \overline{\vt} \ \mbox{ uniformly for } \ h \rightarrow 0 \ \mbox{ and all } \ t \in (0,T).
\end{align*}
Then the family  $\{\vr_h,\vt_h,\vu_h,\Dhuh,\nabla_h\vt_h\}_{h>0}$ generates a Young measure $\{\Nu\}_{(t,x)\in (0,T)\times \Omega}$ that is a
DMV solution of the Navier--Stokes--Fourier system \eqref{ns_eqs} in the sense of Definition~\ref{def_DMV}.
\end{Theorem}

\subsection{Convergence to the strong  solution}

Having shown the family of approximate solutions computed by our finite volume scheme \eqref{scheme_ns} generates the DMV solution of the limit system \eqref{ns_eqs}, we may use the DMV--strong uniqueness principle  established in \cite[Theorem 6.1]{stability_nsf} to get the following result. 

\begin{Theorem} \label{thm_uniqueness}
Let $\kappa > 0$, $\mu > 0$, and $\lambda \geq 0$ be constant. Let the thermodynamic functions $p$, $e$, and $s$ comply with the perfect gas constitutive relations
\[
p(\vr, \vt) = \vr \vt, \ e(\vr, \vt) = c_v \vt,\ s(\vr, \vt) = \log \left( \frac{\vt^{c_v}}{\vr} \right) , \ c_v > 1.
\]
Assume that $\{ \Nu \}_{(t,x) \in (0,T)\times \Omega}$ is a DMV solution of the Navier--Stokes--Fourier system \eqref{ns_eqs} in the sense of Definition~\ref{def_DMV} such that
\begin{equation} \label{C1}
\Nu \left\{ 0 < \underline{\vr} \leq \vr \leq \Ov{\vr}, \ \vt \leq \Ov{\vt}, \ |\vu| \leq \Ov{u} \right\} = 1
\ \mbox{for a.a.} \ (t,x) \in (0,T)\times \Omega
\end{equation}
for some constants $\underline{\vr}$, $\Ov{\vr}$, $\Ov{\vt}$, and $\Ov{\vu}$. Assume further that
\[
\mathcal{V}_{0,x} = \delta_{\vr_0(x), \vt_0(x), \vu_0(x)} \ \ \ \mbox{for a.a.}\ x \in \Omega,
\]
where $(\vr_0, \vt_0, \vu_0)$ belong to the regularity class
\begin{equation}
\label{init_data}
\vr_0, \vt_0 \in W^{3,2}(\Omega), \ \vr, \ \vt > 0 \ \ \mbox{in} \ \ \Omega,\  \vu_0 \in W^{3,2}_0 (\Omega; \mathbb R^3).
\end{equation}
Then
\[
\mathcal{V}_{t,x} = \delta_{ [\tvr(t,x), \tvt(t,x), \tvu(t,x), \mathbf{D}(\tvu) (t,x), \Grad \tvt(t,x)] } \ \mbox{for a.a.}\ (t,x) \in (0,T)\times \Omega,
\]
where $(\tvr, \tvt, \tvu)$ is a strong  (classical) solution to the Navier--Stokes--Fourier system with the initial data $(\vr_0, \vt_0, \vu_0)$.
\end{Theorem}

\begin{proof}

 Under the regularity assumption \eqref{init_data}, the strong solution exists locally in time, say on $[0, T_{\rm max})$, see, e.g., Valli and Zajackowski \cite{Valli}. Thus we can use the DMV--strong uniqueness principle on $[0, T_{\rm max})$. On the other hand, hypothesis \eqref{C1} implies that the no--blow up criterion of Sun, Wang, and Zhang \cite{SunWangZhang} applies yielding $T_{max} = T$.
As a matter of fact, the results of Sun, Wang, and Zhang \cite{SunWangZhang} have been established on a bounded domain with suitable boundary conditions. However their extension to the space--periodic case is straightforward. In particular, the assumption on the uniform bound of the velocity makes it possible to handle general viscosity coefficients (cf. Remark~3 in \cite{SunWangZhang}).
\end{proof}

In order to use the above result we additionally  need that the DMV solution has also bounded velocity, cf. \eqref{assumptions} and \eqref{C1}. Then, as a consequence of Theorem~\ref{thm_uniqueness} and  the DMV--strong uniqueness on $(0,T)\times\Omega,$  we can show that the DMV solution coincides with the global strong solution.

\begin{Theorem}[Convergence to strong solution]\label{thm_convergence}
In addition  to the hypotheses of Theorem~\ref{thm_uniqueness}, suppose that the Navier--Stokes--Fourier system (\ref{ns_eqs}) is endowed with the initial data $(\vr_0,\vt_0, \bfu_0)$ satisfying \eqref{init_data}.
Let $(\vr_h, \vt_h, \vu_h)$ be the solution of the finite volume scheme \eqref{scheme_ns} with $0 < \eps <1, $ satisfying the assumptions \eqref{assumptions} and, in addition,
\begin{equation*}
  |\vu_h (t)| \leq \Ov{u} \mbox{ uniformly for } \ h \rightarrow 0 \ \mbox{ and all } \ t \in (0,T).
\end{equation*}
Then
\begin{align*}
\vrh \rightarrow \vr \mbox{ (strongly) in } L^{p}\left((0,T)\times \Omega \right),\ &\vth \rightarrow \vt \mbox{ (strongly) in }  L^{p}\left((0,T)\times \Omega \right),\\
&\bfu_h \rightarrow \bfu \mbox{ (strongly) in } L^p\left((0,T)\times \Omega; \mathbb R^d \right),\ p \in [1,\infty),
\end{align*}
where $\vr$, $\vt$, and $\bfu$ is a strong (classical) solution of the Navier--Stokes--Fourier system.
\end{Theorem}

\medskip

\begin{Remark}

We have constructed solution having periodic boundary conditions.  When considering a polyhedral domain, the existence of smooth solutions remains open and may be a delicate task. To avoid this problem, one has to approximate a smooth domain by a family of polyhedral domains analogously as in \cite{FHS_16}. Clearly, such a problem does not occur for periodic boundary conditions.
\end{Remark}

\section{Conclusions}
In the present paper we have studied a long--standing open problem of rigorous convergence analysis of finite volume schemes for multidimensional compressible flows. We
have proved that the bounded numerical solutions generated by the finite volume method \eqref{scheme_ns_fv} converge to the global strong solution of the Navier--Stokes--Fourier system (\ref{ns_eqs}) describing motion of viscous compressible and heat conducting
fluids. To this goal we have applied a rather general technique using the dissipative measure--valued solutions. Indeed, realising that for the numerical solutions the conservation of mass \eqref{mass_conservation} and the discrete energy dissipation \eqref{d_en_ineq_old} hold, we have derived
the first {\sl a priori} estimates \eqref{N}. To proceed further
the discrete entropy inequality \eqref{est_entropy}  has played a fundamental role. In order to control the discrete entropy we had to assume boundedness of the discrete density and temperature, cf.~\eqref{assumptions}. This has allowed us, together with the entropy inequality, to obtain the second {\sl a priori} estimates  \eqref{apriori_2.1} and \eqref{apriori_2.2}.
%\eqref{est1}-\eqref{est2}, \eqref{est3}-\eqref{jump_rhot}.
Equipped with the above bounds we have shown in Section~\ref{sec_consistency} the consistency of our finite volume method.

Consequently, the numerical solutions were shown to generate, up
to a subsequence, the Young measure that represents a dissipative measure--valued solution of the Navier--Stokes--Fourier system, see Section~\ref{sec_convergence}.   Using the DMV--strong uniqueness principle, cf.~\cite[Theorem~6.1]{stability_nsf}, we have obtained the strong convergence of the finite volume solutions  towards the strong (classical) solution of the Navier--Stokes--Fourier system (\ref{ns_eqs}) on the lifespan of the latter.
Assuming moreover that the numerical solution emanating from the  initial data satisfying \eqref{init_data} has also bounded velocity, we were able to use the DMV--strong uniqueness result stated in Theorem~\ref{thm_uniqueness} 
 to show the strong convergence  to the global in time strong (classical) solution of \eqref{ns_eqs} without assuming its existence a priori, 
 cf.~Theorem~\ref{thm_convergence}.

 As far as we know this is the first rigorous convergence proof for the
finite volume method applied to the Navier--Stokes--Fourier system.  The numerical flux \eqref{num_flux} in our scheme  is based on the upwinding with an additional numerical diffusion of order $\mathcal{O}(h^{ \eps +1}),$ $0<\varepsilon<1$. In fact, the additional numerical diffusion is only a technical tool. Consequently, our result  implies the convergence of any finite volume method with a numerical diffusion larger than that of our diffusive  upwinding.

\bigskip \noindent
\textbf{Acknowledgement.} E. Feireisl, H. Mizerov\' a and B. She would like to thank DFG TRR 146 Multiscale simulation methods for soft matter systems and the Institute of Mathematics, University Mainz for the hospitality.

%\bibliography{citace}
%\bibliographystyle{plain}

\end{document}